\documentclass[11pt]{article}

\usepackage{amsmath,amssymb,latexsym, amsfonts, amscd, amsthm}
\usepackage[small,nohug,heads=littlevee]{diagrams}
\diagramstyle[labelstyle=\scriptstyle]

\theoremstyle{plain}

\newtheorem{theorem}{Theorem}[section]

\newtheorem{proposition}[theorem]{Proposition}

\newtheorem{lemma}[theorem]{Lemma}

\theoremstyle{definition}
\newtheorem{definition}[theorem]{Definition}

\newcommand{\mat}[4]{\left( \begin{array}{cc} {#1} & {#2} \\ {#3} & {#4}
\end{array} \right)}

\numberwithin{equation}{section}

\usepackage[pdftex,plainpages=false]{hyperref}
\usepackage{titling}
\thanksmarkseries{alph}
\author{Moshe Adrian$^{a,*}$ \\ $^a$Department of Mathematics, University of Utah, Salt Lake City, UT 84112 \\E-mail address: madrian@math.utah.edu \\Telephone number: 1-801-581-6851 \\Fax number: 1-801-581-4148 \\$^*$Corresponding Author}
\begin{document}

\title{A new realization of the Langlands correspondence for $\mathrm{PGL}(2,F)$}
\maketitle




\begin{abstract}
In this paper, we give a new realization of the local Langlands correspondence for $\mathrm{PGL}(2,F)$, where $F$ is a $p$-adic field of odd residual characteristic.   In this case, supercuspidal representations of $\mathrm{PGL}(2,F)$ are parameterized by characters of elliptic tori. Taking a cue from real groups, we propose that supercuspidal representations are
naturally parameterized by characters of covers of tori. Over the reals, Harish-Chandra defined the discrete series representations by specifying their characters restricted to an elliptic torus, and these characters may naturally be expressed in terms of characters of a cover of the torus. We write down a natural analogue of Harish-Chandra's character for $\mathrm{PGL}(2,F)$, and show that it is the character of a unique supercuspidal representation, on a canonical subset of the elliptic torus. This paves the way for a realization of the local Langlands correspondence for $\mathrm{PGL}(2,F)$ that eliminates the need for any character twists.
\end{abstract}

\section{Introduction}

In this paper, we reexamine the local Langlands correspondence for $\mathrm{PGL}(2,F)$, where $F$ is a non-Archimedean local field of characteristic zero with odd residual characteristic, using character theory and ideas from the theory of real reductive groups.  Our main result is a construction of the local Langlands correspondence which gives an alternative explanation of some of the technical issues arising in \cite{bushnellhenniart}.

Our results illuminate some new ideas about character theory of $p$-adic groups and local Langlands correspondence for $p$-adic groups not known before.  In particular, irreducible Weil group representations $W_F \rightarrow \mathrm{SL}(2,\mathbb{C})$ and supercuspidal representations of $\mathrm{PGL}(2,F)$ are most naturally parameterized not by characters of elliptic tori, but by genuine characters of double covers of elliptic tori, as is the case for admissible representations of real groups.  We show that the supercuspidal representations of $\mathrm{PGL}(2,F)$ are also naturally parameterized by genuine characters of double covers of elliptic tori using character theory.  To do this we rewrite supercuspidal characters in terms of double covers of elliptic tori as in Harish-Chandra's discrete series character formula and as in the Weyl character formula.  Rewriting the supercuspidal character formulas in this way paves the way for a construction of local Langlands correspondence for $\mathrm{PGL}(2,F)$ that eliminates the need for any finite order character twists in the local Langlands correspondence for $\mathrm{PGL}(2,F)$ that arise in \cite{bushnellhenniart}.  As we shall see, our results and formulas also give a natural explanation of the character formulas that first appeared in \cite{sallyshalika}.

Let us first recall the classical construction of the local Langlands correspondence for $\mathrm{GL}(2,F)$ (note that the local Langlands correspondence for $\mathrm{PGL}(2,F)$ can be obtained from that of $\mathrm{GL}(2,F)$ by isolating the representations of $\mathrm{GL}(2,F)$ with trivial central character).  We wish to note that a local Langlands correspondence for $\mathrm{GL}(n,F)$ was conjectured by Moy in \cite{moy} for general $n$. Because of recent work of Bushnell and Henniart (see \cite{bushnellhenniart1}), the correspondence of \cite{moy} is indeed correct for $\mathrm{GL}(\ell,F)$, $\ell$ a prime.  By \cite{bushnellhenniart}, there is a map
\begin{equation*}
\left\{
\begin{array}{ll}
isomorphism \ classes \ of \\
admissible \ pairs
\end{array}
\right\} \rightarrow \left\{
\begin{array}{ll}
supercuspidal \ representations \\
of \ \mathrm{GL}(2,F)
\end{array} \right\}
\end{equation*}
$$\hspace{-.5in} (E/F, \chi) \mapsto \pi_{\chi}$$
where $\chi$ is a character of $E^*$ satisfying certain conditions to be described in section \ref{llcreview}, and $E/F$ is an extension of degree $2$.  This map is a bijection (see \cite{bushnellhenniart}).  Moreover, we have a bijection $$\{admissible \ pairs \ (E/F, \chi) \} \rightarrow \{irreducible \ W_F \rightarrow GL(2,\mathbb{C}) \}$$ $$(E/F, \chi) \mapsto Ind_{W_E}^{W_F}(\chi) =: \varphi(\chi)$$ (see \cite{bushnellhenniart}).  The problem is that the obvious map, $$\varphi(\chi) \mapsto \pi_{\chi},$$ the so-called ``naive correspondence'', is not the local Langlands correspondence because $\pi_{\chi}$ has the wrong central character.  Instead, the local Langlands correspondence is given by $$\varphi(\chi) \mapsto \pi_{\chi {}_F \mu_{\chi}}$$ (see \cite{bushnellhenniart}) for some subtle finite order character ${}_F \mu_{\chi}$ of $E^*$.  The proof of this is due to Kutzko (see \cite{kutzko1}), and we will review the main ingredients needed from the correspondence in section \ref{llcreview}.
We will show that if one considers genuine characters of canonical double covers of elliptic tori rather than characters of elliptic tori, then one obtains a realization of the local Langlands correspondence that eliminates the need for any character twists.  We do this in the following way.

Taking the cue from the theory of real groups, we use genuine characters $\tilde\chi$ of certain double covers of elliptic tori, denoted $T(F)_{\tau \circ \rho}$, instead of characters of elliptic tori $T(F)$, to parameterize both supercuspidal Weil parameters for $\mathrm{PGL}(2,F)$ and supercuspidal representations of $\mathrm{PGL}(2,F)$ using character theory.  In section \ref{setup}, we give a method for attaching a genuine character of a double cover of elliptic tori satisfying certain regularity conditions, to a supercuspidal Weil parameter of $\mathrm{PGL}(2,F)$:

\begin{equation}
\{supercuspidal \ W_F \rightarrow SL(2,\mathbb{C}) \} \leftrightarrow \left\{
\begin{array}{ll}
regular \ genuine \ characters \\
of \ T(F)_{\tau \circ \rho}
\end{array} \right\} \label{eq:parameters1}
\end{equation}
$$\hspace{-.75in}\varphi \mapsto \tilde\chi$$ Moreover, given a regular genuine character $\tilde\chi$ of $T(F)_{\tau \circ \rho}$, we write down a conjectural Harish-Chandra type character formula, denoted $F(\tilde\chi)$, in section \ref{setup}.  In sections \ref{theproofofcharacters}, \ref{samecartan}, \ref{othercartan}, and \ref{depthzerochapter}, we show that this naturally gives a bijection

\begin{equation}
\left\{
\begin{array}{ll}
regular \ genuine \ characters \\
of \ T(F)_{\tau \circ \rho}
\end{array}
\right\} \rightarrow \left\{
\begin{array}{ll}
supercuspidal \ representations \\
of \ \mathrm{PGL}(2,F)
\end{array} \right\}. \label{eq:parameters2}
\end{equation}
$$\tilde\chi \mapsto \pi(\tilde\chi)$$ where $\pi(\tilde\chi)$ is the unique supercuspidal representation of $\mathrm{PGL}(2,F)$, whose character, restricted to a certain natural subset of $T(F)$ (to be described later), is $F(\tilde\chi)$.  We remark that the characters of $\mathrm{GL}(2,F)$ were computed by Shimizu (see \cite{shimizu}), and those of $\mathrm{PGL}(2,F)$ by Silberger (see \cite{silberger}).

We then show that the composition of bijections (\ref{eq:parameters1}) and (\ref{eq:parameters2}), $$\varphi \mapsto \tilde\chi \mapsto \pi(\tilde\chi),$$ is the local Langlands correspondence for $\mathrm{PGL}(2,F)$.

Let $G$ be a connected reductive group defined over $\mathbb{R}$, and ${}^L G$ its $L$-group.  In the theory of real groups, an admissible homomorphism $W_{\mathbb{R}} \rightarrow {}^L G$ factors through the normalizer of a torus in ${}^L G$.  As such, it naturally produces a genuine character of the \emph{$\rho$-cover of $T(\mathbb{R})$}, denoted $T(\mathbb{R})_{\rho}$, a certain double cover of $T(\mathbb{R})$, which we now define.  First, we need to make the following definition.

\begin{definition}\label{rhocoverdefinition}
Let $G$ be a connected reductive group over $\mathbb{R}$, and let $T \subset G$ be a maximal torus over $\mathbb{R}$. Let $X^*(T)$ be the character group of $T$.  Let $\Delta^+$ be a set of positive roots of $G$ with respect to $T$.  Let $\rho = \frac{1}{2} \displaystyle\sum_{\alpha \in \Delta^+} \alpha$.  Then $2 \rho \in X^*(T)$.  Viewing $2 \rho$ as also a character of $T(\mathbb{R})$ by restriction, we define the \emph{$\rho$-cover of $T(\mathbb{R})$} as the pullback of the two homomorphisms $$2 \rho : T(\mathbb{R}) \rightarrow \mathbb{C}^* \ \ \ \ \ \ \Upsilon_{\mathbb{R}} : \mathbb{C}^* \rightarrow \mathbb{C}^*.$$ $$t \mapsto 2 \rho(t) \ \ \ \ \ \ \ \ \ \ \ z \mapsto z^2$$  We denote the $\rho$-cover by $T(\mathbb{R})_{\rho}$, and so the following diagram commutes:

$$
\begin{CD}
T(\mathbb{R})_{\rho} @> \rho >> \mathbb{C}^*\\
@VV \Pi_{\mathbb{R}} V @VV \Upsilon_{\mathbb{R}} V\\
T(\mathbb{R}) @>2 \rho>> \mathbb{C}^*
\end{CD}
$$
\end{definition}

Note that because of the commutativity of the diagram, although $\rho$ is not necessarily a character of $T(\mathbb{R})$,  $\rho$ is naturally a character of $T(\mathbb{R})_{\rho}$.  Moreover, $\Pi_{\mathbb{R}}$ is the canonical projection $\Pi_{\mathbb{R}}(t, \lambda) = t$.

The Weyl group acts on $T(\mathbb{R})_{\rho}$ as follows: If $(t, \lambda) \in T(\mathbb{R})_{\rho}$, then define

\begin{equation}
s(t, \lambda) := (st, e^{s^{-1} \rho - \rho}(t) \lambda)  \ \ \forall s \in W(G(\mathbb{R}), T(\mathbb{R})). \ \ \label{eq:weylgroupreal}
\end{equation}

The genuine characters of $T(\mathbb{R})_{\rho}$ that naturally arise from Weil parameters are used to form $L$-packets. One can write down more explicitly the local Langlands correspondence for discrete series representations, and this is the motivation for our work.

For $\mathrm{PGL}(2,\mathbb{R})$, the local Langlands correspondence for discrete series representations is as follows.  Fix a positive set of roots $\Delta^+$ of $G$ with respect to $T$.  Let $\varphi : W_{\mathbb{R}} \rightarrow \mathrm{SL}(2, \mathbb{C})$ be a discrete series parameter, and let $T(\mathbb{R})$ be the compact torus of $\mathrm{PGL}(2,\mathbb{R})$.  Then $\varphi$ naturally gives rise to a genuine character $\tilde\chi$ of $T(\mathbb{R})_{\rho}$.  By Harish-Chandra's discrete series character theorem, $\tilde\chi$ gives rise to a unique discrete series representation, denoted $\pi(\tilde\chi)$, whose character, restricted to the regular elements of $T(\mathbb{R})$, is given by a natural quotient of functions on $T(\mathbb{R})_{\rho}$ (see section \ref{realgroupschapter} for more details).  The map

\begin{equation}
\varphi \mapsto \pi(\tilde\chi)  \label{eq:naivereal}
\end{equation}

\noindent is the local Langlands correspondence for discrete series representations of $\mathrm{PGL}(2,\mathbb{R})$.
Thus, one can write down the correspondence for discrete series in terms of character theory.  This is the approach we take in this paper, and we will show that the above correspondence (\ref{eq:naivereal}) carries over naturally to the $p$-adic setting.  We will review the ingredients that we need from the theory of real groups in section \ref{realgroupschapter}.

One of the results that we will prove is an analogue of Harish-Chandra's discrete series character theorem for $\mathrm{PGL}(2,F)$, where $F$ is a $p$-adic field of characteristic zero.  In doing this, we give a new realization of the  local Langlands correspondence for $\mathrm{PGL}(2,F)$, and the character twists ${}_F \mu_{\chi}$ go away.  Before we present our main results, we need to define the covers of tori that will be essential, which are an analogue of the $\rho$-cover that appears in the theory over the reals.

Let $G$ be a connected reductive group defined over $F$, and $T$ a maximal torus in $G$ defined over $F$.  Let $\Delta^+$ be a choice of positive roots of $G$ with respect to $T$.  Let $\rho = \frac{1}{2} \displaystyle\sum_{\alpha \in \Delta^+} \alpha$.  Let $\lambda$ be a character of $K^*$, where $K/F$ contains the minimal splitting field of $T$. Note that the image of $2 \rho$, restricted to $T(F)$, lies in $K^*$.

\begin{definition}\label{tauofrhocover}
We define the \emph{$\lambda \circ \rho$-cover of $T(F)$}, denoted $T(F)_{\lambda \circ \rho}$, as the pullback of the two homomorphisms $$\lambda \circ 2 \rho : T(F) \rightarrow \mathbb{C}^* \ \ \ \ \ \ \Upsilon : \mathbb{C}^* \rightarrow \mathbb{C}^*$$ $$\ \ \ \ \ t \mapsto \lambda \circ 2 \rho(t) \ \ \ \ \ \ \ z \mapsto z^2.$$

$$
\begin{CD}
T(F)_{\lambda \circ \rho} @>\lambda \circ \rho>> \mathbb{C}^*\\
@VV \Pi V @VV \Upsilon V\\
T(F) @>\lambda \circ 2 \rho>> \mathbb{C}^*
\end{CD}
$$

\noindent That is, $T(F)_{\lambda \circ \rho} = \{(z,w) \in T(F) \times \mathbb{C}^* : \lambda(2 \rho(z)) = w^2 \}.$
\end{definition}

Note that the above map $T(F)_{\lambda \circ \rho} \rightarrow \mathbb{C}^*$ sends $(z, w)$ to $w$.  We have denoted this map by $\lambda \circ \rho$, even though this map is not literally $\lambda$ composed with $\rho$.  Moreover, $\Pi$ is the canonical projection $\Pi(z, w) = z$.  We will use these maps repeatedly.

We need some notation before stating our main results.  If $T$ is a torus in $G$ that is defined over $F$, let $W = W(G(F),T(F))$ denote the relative Weyl group of $G(F)$ with respect to $T(F)$.  Let $G(F) = \mathrm{PGL}(2,F)$, and let $T(F) = E^* / F^*$ be an elliptic torus in $\mathrm{PGL}(2,F)$, where $E = F(\sqrt{\zeta})$ for some $\zeta$.  If $s \in W(G(F),T(F))$ we set $\epsilon(s) := (-1, \zeta)^{\ell(s)}$, where $(,)$ denotes the Hilbert symbol of $F$, where $\ell(s)$ denotes the length of $s$.  Let $T(F)^{reg}$ denote the regular elements of $T(F)$.  Our main results will be the following theorems, which we will prove in sections \ref{positivedepthpaper} and \ref{depthzerochapter}.

\begin{theorem}\label{gl2theorem1}
Assume that the residual characteristic of $F$ is not $2$.  Let $\tilde\tau$ be any character of $E^*$ whose restriction to $F^*$ is $\aleph_{E/F}$, where $\aleph_{E/F}$ is the local class field theory character of $F^*$ relative to $E/F$.  Let $\tau := \tilde\tau \ | \ |_{E}$ where $| \ |_{E}$ denotes the $E$-adic absolute value.  Let $\Delta^+$ be a set of positive roots of G with respect to T.  Let $T(F)_{\tau \circ \rho}$ be as before.  Let $\tilde\chi$ be a genuine character of $T(F)_{\tau \circ \rho}$ that is regular.

Then there exists a unique non-zero constant $\epsilon(\tilde\chi, \Delta^+, \tau)$, depending only on $\tilde\chi, \Delta^+$, and $\tau$, and a unique supercuspidal representation of $\mathrm{PGL}(2,F)$ denoted $\pi(\tilde\chi)$, such that
$$\theta_{\pi(\tilde\chi)}(z) = \epsilon(\tilde\chi, \Delta^+, \tau) \frac{\displaystyle\sum_{s \in W} \epsilon(s)\tilde\chi({}^s w)}{\tau(\Delta^0(z,\Delta^+)) (\tau \circ \rho)(w)} \ \forall z \in T(F)^{reg} \ : 0 \leq n(z) \leq r/2$$ where $w \in T(F)_{\tau \circ \rho}$ is any element such that $\Pi(w) = z$ and $r$ is the depth of $\pi(\tilde\chi)$.  Moreover, every supercuspidal character of $\mathrm{PGL}(2,F)$ is of this form.
\end{theorem}

We will define all of the notation in the above theorem in sections 4-5, including $n(z)$, $\epsilon(\tilde\chi, \Delta^+, \tau)$, and regularity.  We remark that $n(z)$ comes from a canonical filtration on the torus $T(F)$, and is defined in \cite{debacker}.  We wish to make a few comments about the constant $\epsilon(\tilde\chi, \Delta^+, \tau)$.  Firstly, $|\epsilon(\tilde\chi, \Delta^+, \tau)|$ is a known real number in that it has to do with a canonical measure on the Lie algebra.  The subtlety of $\epsilon(\tilde\chi, \Delta^+, \tau)$ is in the fourth roots of unity that appear in it.

There is only one difference between this character formula and the character formula for discrete series of real reductive groups due to Harish-Chandra.  If we were to literally transport the character formula of theorem \ref{Harish-Chandra} to the $p$-adic case, then the denominator $\Delta^0(h, \Delta^+) \rho(\tilde h)$ would take values in $\overline{F}^*$, which would be problematic since characters must take values in $\mathbb{C}^*$.  Therefore, we introduce a $\mathbb{C}^*$-valued character $\tau$ of $\overline{F}^*$ into the denominator in order that the denominator take values in $\mathbb{C}^*$.  In fact, $\tau$ will only need to be a $\mathbb{C}^*$-valued character of $E^*$.

Now let $\varphi$ be a supercuspidal Weil parameter for $\mathrm{PGL}(2,F)$.  We will show in section 4 how to construct a regular genuine character, $\tilde\chi$, of $T(F)_{\tau \circ \rho}$, from $\varphi$.  We will then prove the following theorem.

\begin{theorem}\label{gl2theorem2}
In the setting of theorem \ref{gl2theorem1}, the assignment $$\varphi \mapsto \tilde\chi \mapsto \pi(\tilde\chi)$$ is the Local Langlands correspondence for $\mathrm{PGL}(2,F)$.
\end{theorem}

We wish to make a few comments regarding generalizations.  Firstly, we have shown that everything in this paper can be generalized to $\mathrm{PGL}(n,F)$ and $\mathrm{GL}(n,F)$, $n$ a prime, under the condition that if $n$ is odd, then the residual characteristic of $F$ is greater than $2n$ (see \cite{adrian}). Moreover, jointly with Joshua Lansky, we have been able to generalize this work to more general reductive groups, such as $\mathrm{PGSp}(4,F)$ (see \cite{adrianlansky}), using the theory of ``groups of type $L$'' (see \cite{roe}).  Using the theory of ``groups of type $L$'', we expect to be able to generalize our work to at least the setting of depth zero supercuspidal $L$-packets.  While there is hope for positive depth supercuspidal $L$-packets, the character formulas in the positive depth setting are quite complicated, and we are unsure whether it is possible to unpackage them in a uniform way in order to match them up with our conjectural formulas.

\section{Notation and Definitions}

Let $F$ denote a local non-archimedean field of characteristic zero, $\mathfrak{o}_F$ its ring of integers, and $\mathfrak{p}_F$ the maximal ideal of $\mathfrak{o}_F$. We let $p$ denote a uniformizer of $F$.  Let $k_F$ denote the residue field of $F$ with cardinality $q$. Every element $x \in F^*$ can be written uniquely as $x = u p^n$ for some unit $u \in \mathfrak{o}_F^*$, and some $n \in \mathbb{Z}$.  We set $v_F(x) = n$. We let $W_F$ denote the Weil group of $F$ (see \cite{bushnellhenniart} for its definition).  We choose an element $\Phi \in \mathrm{Gal}(\overline{F}/F)$ whose inverse induces on $\overline{k_F}$ the map $x \mapsto x^q$.  The \emph{level} of a character $\psi : F \rightarrow \mathbb{C}^*$ is defined to be the least integer $d$ such that $\mathfrak{p}_F^d \subset \mathrm{ker}(\psi)$.  The \emph{level} of a character $\chi : F^* \rightarrow \mathbb{C}^*$ is defined to be the least integer $n \geq 0$ such that $\chi$ is trivial on $1 + \mathfrak{p}_F^{n+1}$.  These definitions will be used for field extensions of $F$ as well, where they will continue to be integer valued.  For the definition of \emph{depth} of an irreducible admissible complex representation, see \cite{moyprasad}.  Throughout, we fix once and for all an additive character $\psi$ of $F$ of level one.  If $E/F$ is a separable extension, $N$ (or sometimes $N_{E/F}$) will denote the norm map from $E$ to $F$, $\mathrm{Tr}_{E/F}$ will denote the trace map from $E$ to $F$, and $\mathrm{Aut}(E/F)$ will denote the group of automorphisms of $E$ that fix $F$ pointwise.  When we write a decomposition $w = p^n u$ where $w \in F^*$, we mean that $u \in \mathfrak{o}_F^*$.  If $E/F$ is quadratic and $E = F(\delta)$, we will frequently decompose an element $w \in E$ as $w = p^n u + p^m v \delta$ where we are viewing $E$ as a vector space over $F$ with basis $1, \delta$, and $u,v \in \mathfrak{o}_F^*$.  We will write $\overline{w}$ instead of $\upsilon(w)$ where $1 \neq \upsilon \in Gal(E/F)$. Let $( , )$ denote the Hilbert symbol of $F$.  If $E/F$ is Galois, we let $\aleph_{E/F}$ denote the local class field theory character of $F^*$ relative to the extension $E/F$.  Let $| \ |_F$ denote the standard normalized absolute value on $F$.  If $K/F$ is a finite separable extension, we let $| \ |_K$ denote the unique absolute value of $K$ that extends the absolute value on $F$.  In section 5, $\tilde\tau$ will denote any character of $E^*$ whose restriction to $F^*$ is $\aleph_{E/F}$, where $E/F$ is a tame quadratic extension, and we will set $\tau := \tilde\tau \ | \ |_{E}$.  We will generally write $| \ |$ when it is clear which field we are referring to.

If $$1 \rightarrow \mathbb{Z} / 2 \mathbb{Z} \rightarrow A \rightarrow B \rightarrow 1$$ is an exact sequence of groups, then a character $\beta$ of $A$ is said to be \emph{genuine} if $\beta|_{\mathbb{Z} / 2 \mathbb{Z}}$ is not trivial (that is, $\beta$ does not arise from a character of $B$).  When we say that a $2$-fold cover of a group (as above) \emph{splits}, we mean that the exact sequence splits.  If $G$ denotes any group, then $G^{ab}$ denotes its abelianization.  If $G$ is a connected reductive group defined over $F$ and $T$ is a maximal torus in $G$ defined over $F$, then we will frequently write the finite relative Weyl group as $W$ instead of $W(G(F),T(F))$.  We will write $T(F)^{reg}$ for the set of regular elements in $T(F)$.  In particular, if $T(F) = E^*$, where $E/F$ is a quadratic extension, then the regular elements of $T(F)$ are the set of elements $w \in E^*$ such that $w \neq \overline{w}$.  If $B$ is a normal subgroup of $A$ and $a \in A$, then we will write $[a]$ to denote the class of $a$ in $A/B$.

\section{Discrete series Langlands parameters and character formulas for real groups}\label{realgroupschapter}

In order to motivate the theory that we wish to develop for $p$-adic groups, we describe the corresponding theory over $\mathbb{R}$ since this is what our theory is based upon.  More information can be found in \cite{adamsvogan}.

Let $G$ be a connected reductive group over $\mathbb{R}$ that contains a compact torus.  It is known that this is equivalent to $G(\mathbb{R})$ having discrete series representations.  Recall definition \ref{rhocoverdefinition}.

\begin{definition}
A genuine character $\tilde\chi$ of $T(\mathbb{R})_{\rho}$ is called \emph{regular} if ${}^s \tilde\chi \neq \tilde\chi$ for every $s \in W(G(\mathbb{R}),T(\mathbb{R}))$ such that $s \neq 1$, where ${}^s \tilde\chi(t, \lambda) := \tilde\chi(s^{-1}(t, \lambda))$.
\end{definition}

\begin{definition}
Let $t$ be an indeterminate and let $k$ denote the rank of $G$.  For $h \in G$, define the Weyl denominator $D_G(h)$ by
$$\mathrm{det}(t + 1 - Ad(h)) = D_G(h)t^k + ... (terms \ of \ higher \ order).$$
\end{definition}

\noindent Then if $\Delta$ is the set of roots of $T$ in $G$, $$D_G(h) = \prod_{\alpha \in \Delta} (1 - \alpha(h)).$$

\begin{definition}\label{delta0}
Let $G$ be a connected reductive group over $\mathbb{R}$, $T \subset G$ a maximal torus over $\mathbb{R}$.  Let $\Delta^+$ be a set of positive roots of $G$ with respect to $T$.  Define $$\Delta^0(h, \Delta^+) := \prod_{\alpha \in \Delta^+} (1 - \alpha^{-1}(h)), \ h \in T(\mathbb{R}).$$
\end{definition}

\noindent Let $\rho$ denote half the sum of the positive roots.  Then if the cardinality of $\Delta^+$ is $n$, we have $$(-1)^n D_G(h) = \Delta^0(h, \Delta^+)^2 (2 \rho)(h).$$
Then, if we define $|\rho(h)| := |2 \rho(h)|^{\frac{1}{2}}$ we get that $$|D_G(h)|^{\frac{1}{2}} = |\Delta^0(h, \Delta^+)| |\rho(h)|.$$

In general, $\rho$ is not in $X^*(T)$.  Therefore, $\rho$ is not a well-defined character of $T(\mathbb{R})$.  However, $\rho$ is a well-defined character of $T(\mathbb{R})_{\rho}$.  If $\tilde h \in T(\mathbb{R})_{\rho}$ maps to $h \in T(\mathbb{R})$ via the canonical projection, then $$|D_G(h)|^{\frac{1}{2}} = |\Delta^0(h, \Delta^+)| |\rho(h)| = |\Delta^0(h, \Delta^+)| |\rho(\tilde h)|.$$
We now present the classification of discrete series representations of $G(\mathbb{R})$.

\begin{theorem}\label{Harish-Chandra}(Harish-Chandra)
Let G be a connected reductive group, defined over $\mathbb{R}$.  Suppose that G contains a real Cartan subgroup T that is compact mod center.  Let $\Delta^+$ be a set of positive roots of G with respect to T.  Let $\rho := \frac{1}{2} \displaystyle\sum_{\alpha \in \Delta^+} \alpha$.  Let $\tilde\chi$ be a genuine character of $T(\mathbb{R})_{\rho}$ that is regular.  Let $W := W(G(\mathbb{R}),T(\mathbb{R}))$ be the relative Weyl group of $G(\mathbb{R})$ with respect to $T(\mathbb{R})$.  Let $\epsilon(s) := (-1)^{\ell(s)}$ where $\ell(s)$ is the length of the Weyl group element $s \in W$. Let $T(\mathbb{R})^{reg}$ denote the regular set of $T(\mathbb{R})$.  Then there exists a unique constant $\epsilon(\tilde\chi, \Delta^+) = \pm 1$, depending only on $\tilde\chi$ and $\Delta^+$, and a unique discrete series representation of $G(\mathbb{R})$, denoted $\pi(\tilde\chi)$, such that
$$\theta_{\pi(\tilde\chi)}(h) = \epsilon(\tilde\chi, \Delta^+) \frac{\displaystyle\sum_{s \in W} \epsilon(s) \tilde\chi({}^s \tilde h)}{\Delta^0(h, \Delta^+) \rho(\tilde h)}, \ h \in T(\mathbb{R})^{reg}$$ where $\tilde h \in T(\mathbb{R})_{\rho}$ is any element such that $\Pi(\tilde h) = h$.  Moreover, every discrete series character of $G(\mathbb{R})$ is of this form.
\end{theorem}

This character formula is a variant of the formula found in the literature, using $\rho$-covers.  A reference for this theorem is \cite[Definition 3.1]{adams1}.  We can be more specific about the constant $\epsilon(\tilde\chi, \Delta^+)$.  In particular, $\epsilon(\tilde\chi, \Delta^+) = (-1)^{\ell(s)}$ where $s \in W$ makes $d \tilde\chi$ dominant for $\Delta^+$, $d \tilde\chi$ denoting the differential of $\tilde\chi$.

It is important to note that while the numerator and denominator of the character formula live on $T(\mathbb{R})_{\rho}$, the quotient factors to a legitimate function on $T(\mathbb{R})^{reg}$.

We conclude the section by describing the local Langlands correspondence for discrete series representations of $\mathrm{PGL}(2, \mathbb{R})$. Fix a positive set of roots $\Delta^+$ of $G$ with respect to $T$.  Let $\phi : W_{\mathbb{R}} \rightarrow \mathrm{SL}(2, \mathbb{C})$ be a discrete series Weil parameter.  By the theory in \cite{adamsvogan}, $\phi$ canonically gives rise to a genuine character $\tilde\chi$ of $T(\mathbb{R})_{\rho}$.  By Harish-Chandra's discrete series theorem, $\tilde\chi$ canonically gives rise to a unique discrete series representation, denoted $\pi(\tilde\chi)$, whose distribution character is $$F(\tilde\chi)(h) :=  \epsilon(\tilde\chi, \Delta^+) \frac{\displaystyle\sum_{s \in W} \epsilon(s) \tilde\chi({}^s \tilde h)}{\Delta^0(h, \Delta^+) \rho(\tilde h)}, \ h \in T(\mathbb{R})^{reg}$$ where $\tilde h \in T(\mathbb{R})_{\rho}$ is any element such that $\Pi(\tilde h) = h$.  Then the map $\phi \mapsto \pi(\tilde\chi)$ is the local Langlands correspondence for discrete series representations of $\mathrm{PGL}(2,\mathbb{R})$.  The rest of the paper will be devoted to proving the analogous result for $\mathrm{PGL}(2,F)$, where $F$ is a local non-Archimedean field of characteristic zero.

\section{Positive depth supercuspidal character formulas for $\mathrm{PGL}(2,F)$}\label{positivedepthpaper}

In this section we present our realization of the positive depth local Langlands correspondence for $\mathrm{PGL}(2,F)$, modeled on the local Langlands correspondence for $PGL(2,\mathbb{R})$.  We first briefly review the construction of the local Langlands correspondence for $\mathrm{GL}(2,F)$ as in \cite{bushnellhenniart}, and then we recall some necessary results on the Weil index.  In section \ref{setup}, we attach to a Langlands parameter for $\mathrm{PGL}(2,F)$ a conjectural character formula for a supercuspidal representation of $\mathrm{PGL}(2,F)$.  This formula is written in terms of genuine functions on a double cover of an elliptic torus.  In that section, we will unwind the formula so that it can be written in terms of functions on an elliptic torus, so that we can compare it with the known character formula of a supercuspidal representation of $\mathrm{PGL}(2,F)$ (which is written in terms of functions on an elliptic torus).  In particular, we will eventually pull back our character formula to a character formula on $\mathrm{GL}(2,F)$, and then compare this formula with the known character formula of a supercuspidal representation of $\mathrm{GL}(2,F)$ (recall that the representations of $\mathrm{PGL}(2,F)$ are the representations of $\mathrm{GL}(2,F)$ with trivial central character).  In section \ref{theconstantepsilon}, we recall the formula for the character of a positive depth supercuspidal representation of $\mathrm{GL}(2,F)$, and we define the constant $\epsilon(\tilde\chi, \Delta^+, \tau)$ that appears in our character formula in theorem \ref{gl2theorem1}.  In section \ref{decompositions}, we recall the Moy-Prasad filtrations on an elliptic torus of $\mathrm{GL}(2,F)$, and we then explicitly determine the depth zero elements in such an elliptic torus.  In section \ref{theproofofcharacters}, we prove that the conjectural character formula that we attached to a Langlands parameter for $\mathrm{PGL}(2,F)$ agrees, far away from the identity, with the character formula of the supercuspidal representation that is attached to the Langlands parameter via the local Langlands correspondence (in our setting, if $E^*$ is an elliptic torus in $\mathrm{GL}(2,F)$, then the elements in $E^*$ that are \emph{far away from the identity} are just the elements in $E^* \setminus F^* (1 + \mathfrak{p}_E)$.  The rest of section \ref{positivedepthpaper} is devoted to showing that there are no other supercuspidal representations of $\mathrm{PGL}(2,F)$ whose character agrees with ours, far away from the identity.

\subsection{Local Langlands correspondence for $\mathrm{GL}(2,F)$}\label{llcreview}

In this section, we describe the construction of the local Langlands correspondence for $\mathrm{GL}(2,F)$ as explained in \cite{bushnellhenniart}.

Let $E/F$ be a tamely ramified quadratic extension and $\chi$ a character of $E^*$.  Recall that $N$ denotes the norm map from $E$ to $F$.

\begin{definition}
The pair $(E/F, \chi)$ is called an \emph{admissible pair} if

(i) $\chi$ does not factor through $N$ and

(ii) If $\chi|_{1 + \mathfrak{p}_E}$ factors through $N$, then $E/F$ is unramified.

\end{definition}

We write $\mathbb{P}_2(F)$ for the set of $F$-isomorphism classes of admissible pairs.  For more information about admissible pairs, see \cite[Section 18]{bushnellhenniart}.

\begin{definition}\ref{minimaladmissible}
Let $(E/F, \chi)$ be an admissible pair where $\chi$ is level $n$.  We say that $(E/F, \chi)$ is \emph{minimal} if $\chi|_{1 + \mathfrak{p}_E^n}$ does not factor through $N$.
\end{definition}

Let $\mathbb{A}_2^0(F)$ denote the set of equivalence classes of all irreducible supercuspidal representations of $\mathrm{GL}(2,F)$.  We then have the following theorem:

\begin{theorem}{\label{positivedepth}}{\cite[Section 20.2]{bushnellhenniart}}
Suppose that the residual characteristic of $F$ is not $2$.  There is a map $(E/F, \chi)$ $\mapsto \pi_{\chi}$ which induces a bijection $$\mathbb{P}_2(F) \rightarrow \mathbb{A}_2^0(F).$$

Furthermore, if $(E/F, \chi) \in \mathbb{P}_2(F)$, then:

(i) $\omega_{\pi_{\chi}} = \chi|_{F^*}$

(ii) if $\phi$ is a character of $F^*$, then $\pi_{\chi \phi_E} = \phi \pi_{\chi}$.
\end{theorem}

Let $\mathbb{G}_2^0(F)$ be the set of equivalence classes of irreducible smooth two-dimensional representations of $W_F$. Recall that there is a local Artin reciprocity isomorphism $W_E^{ab} \cong E^*$.  Then, if $(E/F, \chi) \in \mathbb{P}_2(F)$, $\chi$ gives rise to a character of $W_E^{ab}$, which we can pullback to a character, also denoted $\chi$, of $W_E$.  We can then form the induced representation $\phi_{\chi} = Ind_{W_E}^{W_F} \chi$ of $W_F$.

If $K/F$ is a finite separable extension and $\psi$ is a character of $F$ where $\psi \neq 1$, let $\lambda_{K/F}(\psi)$ denote the Langlands constant, as in \cite[Section 34.3]{bushnellhenniart}.

\begin{definition}
Let $(E/F, \chi)$ be an admissible pair in which $E/F$ is unramified.  Define ${}_F \mu_{\chi}$ to be the unique quadratic unramified character of $E^*$.
\end{definition}

Let $\mu_F$ denote the group of roots of unity in $F$ of order prime to the residual characteristic of $F$. Let $E/F$ be a totally tamely ramified quadratic extension, let $\varpi$ be a uniformizer of $E$, and let $\beta \in E^*$.  Since $\mathfrak{o}_E^* = \mu_E (1 + \mathfrak{p}_E) = \mu_F (1 + \mathfrak{p}_E)$, there is a unique root of unity $\Gamma(\beta, \varpi) \in \mu_F$ such that $$\beta \varpi^{-v_E(\beta)} = \Gamma(\beta, \varpi) \ (\mathrm{mod} \ (1 + \mathfrak{p}_E)).$$

\begin{definition}

(i) Let $(E/F, \chi) \in \mathbb{P}_2(F)$ be a minimal admissible pair such that $E/F$ is totally tamely ramified.  Let $n$ be the level of $\chi$ and let $\alpha(\chi) \in \mathfrak{p}_E^{-n}$ satisfy $\chi(1+x) = \psi_E(\alpha(\chi) x), x \in \mathfrak{p}_E^n$.  There is a unique character ${}_F \mu_{\chi}$ of $E^*$ such that:

$${}_F \mu_{\chi}|_{1 + \mathfrak{p}_E} = 1, \ \ {}_F \mu_{\chi}|_{F^*} = \aleph_{E/F},$$ $${}_F \mu_{\chi}(\varpi) = \aleph_{E/F}(\Gamma(\alpha(\chi), \varpi)) \lambda_{E/F}(\psi)^n.$$

(ii) Let $(E/F, \chi) \in \mathbb{P}_2(F)$ and suppose that $E/F$ is totally tamely ramified.  Write $\chi = \chi' \chi_E$ for a minimal admissible pair $(E/F, \chi')$ and a character $\chi$ of $F^*$.  Define $${}_F \mu_{\chi} = {}_F \mu_{\chi'}.$$
\end{definition}

\begin{theorem}{$\mathbf{Local \ Langlands \ Correspondence}$}{\label{tamellc}}{\cite[Section 34]{bushnellhenniart}}

Suppose the residual characteristic of $F$ is not $2$.  For $\phi \in \mathbb{G}_2^0(F)$, define $\pi(\phi) = \pi_{\chi {}_F \mu_{\chi}}$ for any $(E/F, \chi) \in \mathbb{P}_2(F)$ such that $\phi \cong \phi_{\chi}$.  The map $$\pi : \mathbb{G}_2^0(F) \rightarrow \mathbb{A}_2^0(F)$$ is the local Langlands correspondence for $\mathrm{GL}(2,F)$.
\end{theorem}

\begin{proposition}\label{centralcharacter}{\cite[Section 33]{bushnellhenniart}}
If $\phi \in \mathbb{G}_2^0(F)$ and $\pi = \pi(\phi)$, then $\omega_{\pi} =$ $\mathrm{det}(\phi)$.
\end{proposition}

\subsection{The Weil index}\label{variousneededconstants}

Here we collect some basic results on the Weil index, which we will need for various computations.  These results are due to Ranga Rao, and one can find the proofs and definitions in \cite{rao}.  Throughout, $F$ denotes either a local field or a finite field of characteristic $\neq 2$.  Let $\eta$ be a nontrivial additive character of $F$.  For any $a \in F$, we write $a \eta$ for the character $a \eta : x \mapsto \eta(ax)$.

\begin{definition}\label{rao0}
Define $$\gamma_F(\eta) := \mathrm{Weil \ index \ of} \ x \mapsto \eta(x^2)$$ $$\gamma_F(a, \eta) := \gamma_F(a \eta) / \gamma_F(\eta) \ \ \ a \in F^*.$$
\end{definition}

\begin{lemma}\label{rao1}

$(1) \ \gamma_F(a c^2, \eta) = \gamma_F(a, \eta)$ and $\gamma_F(ab, \eta) \gamma_F(a, \eta)^{-1} \gamma_F(b, \eta)^{-1} = (a,b)$.

$(2) \ \gamma_F(-1, \eta) = \gamma_F(\eta)^{-2}$.

$(3) \ \{\gamma_F(a, \eta) \}^2 = (-1, a) = (a,a)$.
\end{lemma}

Let $Q$ be a nondegenerate quadratic form of degree $n$ over $F$.

\begin{definition}\label{rao2}
The Hasse invariant $h_F(Q)$ is defined as follows:  $$h_F(Q) = \gamma(\eta \circ Q) \{\gamma_F(\eta) \}^{-n} \{\gamma_F(\mathrm{det} Q, \eta) \}^{-1}.$$ Here $\gamma(\eta \circ Q)$ is the Weil index of $x \mapsto \eta(Q(x))$.
\end{definition}

\begin{lemma}\label{rao3}
If $n = 2$, and $Q = a_1 x_1^2 + a_2 x_2^2, a_1, a_2 \in F^*$, then $h_F(Q) = (a_1, a_2)$.
\end{lemma}

Let $\zeta \in \mathfrak{o}_F^*$ be a non-square, and let $\psi$ be an additive character of $F$.  Let $\ell(\psi)$ denote the level of $\psi$.  We will need the following two results when we match up our conjectural character formula with a supercuspidal character of $\mathrm{PGL}(2,F)$.

\begin{lemma}\label{sponge}
$\gamma_F(\zeta, \psi) = (-1)^{\ell(\psi)}$.
\end{lemma}

\proof
Let $a \in F^*$ be arbitrary.  By definition \ref{rao0}, $\gamma_F(a, \psi) = \frac{\gamma_F(a \psi)}{\gamma_F(\psi)}$.  Proposition A.11 of \cite{rao} implies that
$$\frac{\gamma_F(a \psi)}{\gamma_F(\psi)} = \frac{
\gamma_{\overline{F}}( \overline{a \psi})^{\ell(a \psi)} } {
\gamma_{\overline{F}}( \overline{\psi})^{\ell(\psi)} }$$ where
$\ell(\psi)$ is the level of $\psi$ and where $\overline{F}$ is the residue field of $F$.

Now, $$\gamma_{\overline{F}}( \overline{a \psi})^{\ell(a \psi)} =
\gamma_{\overline{F}}(\overline{a}, \overline{\psi})^{\ell(a \psi)}
\gamma_{\overline{F}}(\overline{\psi})^{\ell(a \psi)}$$ by definition \ref{rao0}.  We therefore have

$$\frac{ \gamma_{\overline{F}}( \overline{a \psi})^{\ell(a \psi)} }
{ \gamma_{\overline{F}}( \overline{\psi})^{\ell(\psi)} } = \frac{
\gamma_{\overline{F}}(\overline{a}, \overline{\psi})^{\ell(a \psi)}
\gamma_{\overline{F}}(\overline{\psi})^{\ell(a \psi)}  }{
\gamma_{\overline{F}}(\overline{\psi})^{\ell(\psi)}  }.$$

Write $a = p^m u$.  Then $\ell(a \psi) = \ell(\psi) - m$.  Moreover, $$\gamma_{\overline{F}}(\overline{a}, \overline{\psi}) = \left(\frac{\overline{a}}{\overline{F}}\right)$$   by proposition A.9 of \cite{rao}, where $\left(\frac{\overline{a}}{\overline{F}}\right)$ denotes the Legendre symbol.

In our setting, $a = \zeta$ is a non-square unit, so $m = 0$.  Therefore, altogether, we have $$\gamma_F(\zeta, \psi) = \left(\frac{\overline{\zeta}}{\overline{F}}\right)^{\ell(\zeta
\psi)} = (-1)^{\ell(\psi)}. \qedhere $$

Now let $\lambda_{E/F}(\psi)$ be the Langlands constant from \cite[pages 216, 217, 241]{bushnellhenniart}, where $E/F$ is any tamely ramified quadratic extension.  Write $E = F(\sqrt{\zeta})$, for some $\zeta$.
\begin{lemma}\label{psilevel}
$\lambda_{E/F}(\psi) = \gamma_F(\zeta, \psi) (-1, \zeta)$.
\end{lemma}

\proof
By \cite[page 240-241]{bushnellhenniart}, $\lambda_{E/F}$ is the Weil index (cf. definition \ref{rao0}) of $$q : E^* \rightarrow \mathbb{C}^*$$ where $q(z) := \psi(N(z))$, where $N$ is the norm map $N : E^* \rightarrow F^*$.  The associated symmetric bilinear form is $$(z,w)_q := \frac{\mathrm{Tr}(z \overline{w})}{2}.$$  We wish to therefore calculate the Weil index $\gamma(\psi \circ N)$, which will be the Langlands constant $\lambda_{E/F}(\psi)$.

A basis for $E/F$ is $1, \delta$, where $E = F(\delta), \delta = \sqrt{\zeta}$.  Then, $(1,1)_q= 1$, $(1, \delta)_q = (\delta, 1)_q = 0,$ and  $(\delta, \delta)_q = -\zeta$.  Thus, the matrix of the quadratic form is $$\mat{1}{0}{0}{-\zeta}.$$  Then, by definition \ref{rao2}, $$\gamma(\psi \circ N) = h_F(N) \gamma_F(\psi)^n \gamma_F(\mathrm{det}(N),\psi).$$

In our setting, $n=2$, and by lemmas \ref{rao1} and \ref{rao3}, and since $(-1,-1) = 1$, we have $$\gamma(\psi \circ N) = (1, -\zeta) \gamma_F(\psi)^2 \gamma_F(-\zeta, \psi) = \gamma_F(-1, \psi)^{-1} \gamma_F(-\zeta, \psi) =$$ $$ \gamma_F(-1, \psi) \gamma_F(-1, \psi)^{-2} \gamma_F(-\zeta, \psi) = \gamma_F(-1, \psi) (-1, -1) \gamma_F(-\zeta, \psi) =$$ $$ \gamma_F(\zeta, \psi) (-1, -\zeta) = \gamma_F(\zeta, \psi) (-1, \zeta). \qedhere $$

\subsection{Setup}\label{setup}

In this section we will prove theorem \ref{gl2theorem1} and theorem \ref{gl2theorem2} for the positive depth supercuspidal representations of $\mathrm{PGL}(2,F)$.

Let $F$ denote a non-Archimedean local field of characteristic zero with residual characteristic coprime to $2$.  Let $E/F$ be a quadratic extension.  Write $E = F(\sqrt{\zeta})$ for some $\zeta \in F^*$ and let $\delta := \sqrt{\zeta}$.  $E^*$ embeds as an elliptic torus in $\mathrm{GL}(2,F)$ via the map $$E^* \hookrightarrow \mathrm{GL}(2,F)$$ $$a + b \delta \mapsto \mat{a}{b}{b \zeta}{a}$$ and therefore $E^*/F^*$ embeds in $\mathrm{PGL}(2,F)$ as an elliptic torus as well.

We now explain why double covers of tori play a role.  We start by considering the group $\mathrm{PGL}(2,F)$.  First recall that the representations of $\mathrm{PGL}(2,F)$ are precisely the representations of $\mathrm{GL}(2,F)$ with trivial central character.  Let $\phi$ be a supercuspidal Weil parameter for $\mathrm{PGL}(2,F)$.  Then $\phi = Ind_{W_E}^{W_F}(\chi)$, for some admissible pair $(E/F, \chi)$, by \cite[Theorem 34.1]{bushnellhenniart}.  It is a fact (see \cite[Proposition 29.2]{bushnellhenniart}) that $\mathrm{det}(Ind_{W_E}^{W_F}(\chi)) = \chi|_{F^*} \otimes \delta_{E/F}$, where $\delta_{E/F} = \mathrm{det}(Ind_{W_E}^{W_F}(1))$.  In the case that $E/F$ is quadratic, $\delta_{E/F} = \aleph_{E/F}$.  Thus, by proposition \ref{centralcharacter}, $\chi|_{F^*} = \aleph_{E/F}$. Therefore, the supercuspidal representations of $\mathrm{PGL}(2,F)$ are parameterized by the admissible pairs $(E/F, \chi)$ where $\chi|_{F^*} = \aleph_{E/F}$.  Such a $\chi$ is not a character of the elliptic torus $E^* / F^*$ in $\mathrm{PGL}(2,F)$.  Rather, it is a genuine character of a double cover $E^* / N(E^*)$ of $E^* / F^*$ arising from the canonical exact sequence (note that $ker(\aleph_{E/F}) = N(E^*)$): $$1 \longrightarrow F^* / N(E^*) \longrightarrow E^* / N(E^*) \rightarrow E^* / F^* \longrightarrow 1.$$ Since $F^* / N(E^*) \cong \mathbb{Z}/ 2 \mathbb{Z}$ by Local Class Field Theory, we have that $E^* / N(E^*)$ is a double cover of the torus $E^* / F^*$.  Then the character $\chi$ of $E^*$ naturally factors to a genuine character $\tilde\chi$ of $E^* / N(E^*)$, given by $\tilde\chi([w]) := \chi(w) \ \forall [w] \in E^* / N(E^*)$.  Therefore, the supercuspidal representations of $\mathrm{PGL}(2,F)$ are naturally parameterized by genuine characters of the double cover $E^* / N(E^*)$.  This double cover $E^* / N(E^*)$ is none other than an analogue of the $\rho$-cover that appears in the theory over the reals (see definition \ref{rhocoverdefinition}).  We explain this now.

Let $T(F)$ be an elliptic torus in $\mathrm{PGL}(2,F)$.  Then $T(F) \cong E^* / F^*$ for some quadratic extension $E/F$.  A choice of isomorphism $T(F) \cong E^* / F^*$ defines a standard positive root $\alpha(z) = z / \overline{z}$, for $z \in E^* / F^*$. Let $\rho$ be half the positive root.  Fix a character $\tilde\tau$ of $E^*$ whose restriction to $F^*$ is $\aleph_{E/F}$, and set $\tau := \tilde\tau | \ |_E$.  Recall the denominator that was defined in theorem \ref{gl2theorem1}.  Although $\tau \circ \rho$ is not naturally a function on $E^* / F^*$ since in particular $\rho$ is not naturally a function on $E^* / F^*$, it is by definition a function on the $\tau \circ \rho$-cover of $E^* / F^*$.   Recall definition \ref{tauofrhocover}.  Then in our case, $T(F)_{\tau \circ \rho} = \{(z,\lambda) \in E^* / F^* \times \mathbb{C}^* : \tau(2 \rho(z)) = \lambda^2 \}$.

\begin{lemma}\label{doublecovertoriisomorphism}
$E^* / N(E^*) \cong T(F)_{\tau \circ \rho}$.
\end{lemma}

\begin{proof}
Define the map $$E^* / N(E^*) \xrightarrow{\kappa} T(F)_{\tau \circ \rho}$$ $$[w] \mapsto (z, \tilde\tau(w) |2 \rho(z)|^{1/2})$$ where we are taking the positive square root of the absolute value, and where $z$ is the image of $[w] \in E^* / N(E^*)$ in $E^* / F^*$. To show injectivity, suppose $\kappa([w]) = (1, 1)$, where $w \in E^*$.  Since $z = 1$, we get $w \in F^*$.  But since $\tilde\tau(w) |2 \rho(z)|^{1/2} = 1$, we get that $w \in N(E^*)$ since $\tilde\tau|_{F^*} = \aleph_{E/F}$ and since $|2 \rho(z)| = 1$ since $z = 1$.  To show surjectivity, suppose $(z, \lambda) \in T(F)_{\tau \circ \rho}$, where $z$ is the image of $w \in E^*$ in $E^* / F^*$.  Then, by definition of $T(F)_{\tau \circ \rho}$, we get that $\tau(2 \rho(z)) = \lambda^2$.  This means that $\tilde\tau(w / \overline{w}) |2 \rho(z)| = \lambda^2$.  But $\tilde\tau$ is trivial on the norms, so we have that $\tilde\tau(w / \overline{w}) = \tilde\tau(w^2 / N(w)) = \tilde\tau(w)^2$.  Therefore, $\lambda = \pm \tilde\tau(w) |2 \rho(z)|^{1/2}$.  If $\lambda = \tilde\tau(w) |2 \rho(z)|^{1/2}$, then we get that $\kappa([w]) = (z, \lambda)$.  If $\lambda = - \tilde\tau(w) |2 \rho(z)|^{1/2}$, then let $x \in F^* \setminus N(E^*)$.  Then $\kappa([x w]) = (z, \lambda)$.  Therefore, $\kappa$ is surjective.  Since $\kappa$ is clearly a homomorphism, $\kappa$ is an isomorphism.
\end{proof}

We note that the importance of the term $|2 \rho(z)|^{1/2}$ comes from the fact that $$|D(z)|^{\frac{1}{2}} = |\Delta^0(z, \Delta^+)| |2 \rho(z)|^{1/2},$$ an observation made in section 3.  The reason why this is important is that the term $|D(z)|^{1/2}$ appears in the supercuspidal characters (see section \ref{theconstantepsilon}).

Now let's write down the character formula for a supercuspidal representation of $\mathrm{PGL}(2,F)$.  Let $\phi : W_F \rightarrow GL(2,\mathbb{C})$ be a supercuspidal parameter for $\mathrm{PGL}(2,F)$ so that $\phi = Ind_{W_E}^{W_F}(\chi)$ for some admissible pair $(E/F, \chi)$, where $E = F(\sqrt{\zeta})$ for some $\zeta$. As discussed earlier, this gives us a genuine character $\tilde\chi$ of $E^* / N(E^*)$.  Now recall from theorem \ref{gl2theorem1} the proposed character formula for $F(\tilde\chi)$.  We naturally constructed a genuine character $\tilde\chi$ of $E^* / N(E^*)$.  However, the functions in $F(\tilde\chi)$ have domain $T(F)_{\tau \circ \rho}$.  Recall that $T(F)_{\tau \circ \rho} \cong E^* / N(E^*)$ by lemma \ref{doublecovertoriisomorphism}, so we can pull the function $(\tau \circ \rho)(w)$ and the Weyl group action in $F(\tilde\chi)$ back to $E^* / N(E^*)$ via this isomorphism, and leave our constructed $\tilde\chi$ as living on $E^* / N(E^*)$.  That is, we consider $$F(\tilde\chi)(z) = \epsilon(\tilde\chi, \Delta^+, \tau) \frac{\displaystyle\sum_{s \in W} \epsilon(s)\tilde\chi({}^s [w])}{\tau(\Delta^0(z,\Delta^+)) (\tau \circ \rho)(\kappa([w]))}, \ \ z \in T(F)^{reg}$$ where $[w] \in E^* / N(E^*)$ such that $\Pi(\kappa([w])) = z$.  We will define $\epsilon(\tilde\chi, \Delta^+, \tau)$ in section \ref{theconstantepsilon}.  Unwinding the definitions, we see that $$(\tau \circ \rho)(\kappa([w])) = \tilde\tau(w) |2 \rho(z)|^{1/2} \ \forall [w] \in E^* / N(E^*),$$ where $z$ is the image in $E^* / F^*$ of $[w]$.

We also need to define the Weyl group action.  The Weyl group action on the $\tau \circ \rho$-cover is obtained as follows. If $(z,\lambda)$ is an element of $T(F)_{\tau \circ \rho}$, then analogously to the real case (recall (\ref{eq:weylgroupreal})), define $s(z,\lambda) = (sz,\lambda \tau((s^{-1} \rho - \rho)(z)))$ for $s \in W = W(G(F),T(F)) = \mathrm{Aut}(E/F)$, the relative Weyl group.  Note that this is well-defined.  Simplifying this expression, we get $s(z,\lambda) = (\overline{z}, \lambda \tau(\overline{w}/w))$ when $s \in W$ is nontrivial, and where $\overline{z}$ means the image of $\overline{w}$ in $E^* / F^*$.  Then, since our character formula lives on $E^* / N(E^*)$, we must pull back this action from $T(F)_{\tau \circ \rho}$ to $E^* / N(E^*)$ via $\kappa$.  Doing this, we see that we get $s [w] = \kappa^{-1}(s \kappa([w])) = \kappa^{-1}(s (z, \tilde\tau(w) |2 \rho(z)|^{1/2})) = \kappa^{-1}(\overline{z}, \tilde\tau(w) |2 \rho(z)|^{1/2} \tau(\overline{w}/w)) = \kappa^{-1}(\overline{z}, \tilde\tau(\overline{w}) |2 \rho(\overline{z})|^{1/2}) = \overline{z} \forall [w] \in E^* / N(E^*)$ when $s \in W = \mathrm{Aut}(E/F)$ is nontrivial, since $|2 \rho(z)| = |w / \overline{w}| = 1 \ \forall w \in E^*$, where $z$ is the image of $w$ in $E^* / F^*$.

Then, pulling the denominator of the proposed character formula and the Weyl group action back to $E^* / N(E^*)$ via $\kappa$, we get $$F(\tilde\chi)(z) = \epsilon(\tilde\chi, \Delta^+, \tau) \frac{\tilde\chi([w]) + (-1, \zeta) \tilde\chi([\overline{w}])}{\tilde\tau(1-\overline{z}/z) |\Delta^0(z, \Delta^+)| \tilde\tau(w) |2 \rho(z)|^{1/2} }$$ where $z \in E^*/F^*$ and $[w] \in E^* / N(E^*)$ is some element that maps to $z$ under the map $E^* / N(E^*) \rightarrow E^* / F^*$.  We can also pull this character formula all the way back to $E^*$, and we get $$F(\tilde\chi)(z) = \epsilon(\tilde\chi, \Delta^+, \tau) \frac{\chi(w) + (-1, \zeta) \chi(\overline{w})}{\tilde\tau(1- \overline{w}/w) |\Delta^0(w, \Delta^+)| \tilde\tau(w) |2 \rho(w)|^{1/2}}$$ where $z \in E^*/F^*$ and $w \in E^*$ is some element that maps to $z$ under the map $E^* \rightarrow E^* / F^*$. We will see that this proposed character formula for $\mathrm{PGL}(2,F)$ is independent of the choice of $\tau$.

Note that our formula simplifies: $$F(\tilde\chi)(z) = \epsilon(\tilde\chi, \Delta^+, \tau) \frac{\chi(w) + (-1, \zeta) \chi(\overline{w})}{\tilde\tau(w-\overline{w}) |D(z)|^{1/2}}, \ z \in T(F)^{reg}$$ where $w \in E^*$ is any element that maps to $z \in T(F) = E^* / F^*$ under the canonical map $E^* \rightarrow E^* / F^*$.

The reason is that if $[w] \in E^* / N(E^*)$ maps to $z \in E^* / F^*$, then $2 \rho(z) = w / \overline{w}$, since the positive root of $\mathrm{GL}(2,F)$ factors to $\mathrm{PGL}(2,F)$ and since $|D(z)|^{1/2} = |\Delta^0(z, \Delta^+)| |2 \rho(z)|^{1/2}$ from section 3.

We also note that if we had made the other choice of $\Delta^+$, the denominator in our character formula would include the term $\tilde\tau(\overline{w} - w)$ instead of $\tilde\tau(w - \overline{w})$.  However, because our definition of $\epsilon(\tilde\chi, \Delta^+, \tau)$ includes the term $\epsilon(\Delta^+)$ (see section \ref{theconstantepsilon}), our overall character formula $F(\tilde\chi)$ remains the same regardless of the choice of positive root.

Summing up, noting that $T(F)_{\tau \circ \rho} \cong E^* / N(E^*)$, then we have given a method of assigning a conjectural character formula for a supercuspidal representation of $\mathrm{PGL}(2,F)$, to a supercuspidal Weil parameter of $\mathrm{PGL}(2,F)$, given by

\begin{eqnarray}
\{ supercuspidal \ \phi : W_F \rightarrow SL(2,\mathbb{C}) \} & \mapsto & \tilde\chi \in \widehat{T(F)}_{\tau \circ \rho} \mapsto F(\tilde\chi). \nonumber
\end{eqnarray}

\subsection{The constant $\epsilon(\tilde\chi, \Delta^+, \tau)$}\label{theconstantepsilon}

We now turn to the question of defining the constant $\epsilon(\tilde\chi, \Delta^+, \tau)$.  We recall the main theorem of \cite{debacker} which describes the distribution characters of positive depth supercuspidal representations of $\mathrm{GL}(2,F)$.  We only write down the part of the character formula that we need.

\begin{theorem}{\label{DeBacker1}}{\cite[Theorem 5.3.2]{debacker}}  Let $(E/F, \phi)$ be an admissible pair where $E/F$ has degree $2$ and $\phi$ has positive level, and write $G' = E^*$  Let $\pi = \pi_{\phi}$ be the associated supercuspidal representation of $\mathrm{GL}(2,F)$ given by theorem \ref{positivedepth}.  Then

\begin{equation*}
\frac{\theta_{\pi}(w)}{\mathrm{deg}(\pi)} = \left\{
\begin{array}{rl}
C \lambda(\sigma) \displaystyle\sum_{\overline{w} \in W} \phi({}^w t) & \text{if } n(w) = 0 \ and \ w = {}^g t \ with \ g \in G \\
&  and \ t \in G'\\
C \displaystyle\sum_{\overline{w} \in W} \phi({}^w t) \gamma(\alpha(\phi), {}^w Y) & \text{if } 0 < n(w) \leq r/2 \ and \ w = {}^g t \\
&  where \ t = c(1 + Y) \ with \  c \in Z, g \in G, \\
& and \ Y \in \mathfrak{g}_{n(w)}' \setminus (\mathfrak{z}_{n(w)} + \mathfrak{g}_{n(w)^+}')
\end{array} \right.
\end{equation*}

\end{theorem}

We note that in \cite{debacker}, $X_\pi$ is the notation used instead of $\alpha(\phi)$ (recall the notation $\alpha(\phi)$ from section \ref{llcreview}).  We prefer to use the notation $\alpha(\phi)$.  Moreover, $C := c_{\psi}(\mathfrak{g'}) c_{\psi}^{-1}(\mathfrak{g}) |D(w)|^{-1/2} |\eta(\alpha(\phi))|^{-1/2}$ is defined in \cite[Section 5.3]{debacker}.  This notation is explained in \cite[Sections 4-5]{debacker}.  In particular, $D$ and $\eta$ denote the Weyl discriminants of $G$ and $\mathfrak{g}$, respectively.  Moreover, $n(w)$ will be defined in section \ref{decompositions}.

Let $(E/F, \chi)$ be an admissible pair.  To define $\epsilon(\tilde\chi, \Delta^+, \tau)$, we need to calculate the constant $\gamma(\alpha(\chi), Y)$ in the above theorem.  We refer the reader to \cite[Theorem 5.3.2]{debacker} and \cite[page 55]{debacker} for various notation.  We follow the same notation.

Let $G = \mathrm{GL}(2,F), \mathfrak{g} = \mathfrak{gl}(2,F), G' = E^*, \mathfrak{g}' = E$.  We have a direct sum decomposition $\mathfrak{g} = \mathfrak{g}' + \mathfrak{g}^{\bot}$ where the perpendicular is taken with respect to the trace form $\langle Z_1, Z_2 \rangle := \mathrm{Tr}(Z_1 Z_2) \ \forall Z_1, Z_2 \in \mathfrak{g}$.  Recall our fixed additive character $\psi$ of $F$.  Let $V,W \in \mathfrak{g}^{\bot}$.  Then we define $$Q_{(\alpha(\chi),Y)}(V,W) := (\mathrm{trace}([\alpha(\chi),W][V,Y])/2).$$ $Q_{(\alpha(\chi),Y)}(V,W)$ is a non-degenerate, symmetric, bilinear form on $\mathfrak{g}^{\bot}$.  Then, $\gamma(\alpha(\chi), Y)$ is by definition the Weil Index of $\psi \circ Q_{(\alpha(\chi),Y)}$ (see \cite{rao}).

Recall that $\mathfrak{g}'$ may be embedded in $M_2(F)$ by

$$\mathfrak{g}' \hookrightarrow \mathfrak{g}$$ $$\ \ \ \ \ \ \ \ \ \ \ \ \ \ \ \ \ \ \ \ \ \ \ \ a + d \delta \mapsto \mat{a}{d}{d \zeta}{a}, \ \ a,d \in F$$ where $E = F(\delta)$ with $\delta = \sqrt{\zeta}$.

\begin{lemma}\label{weilcalculation}
Let $\alpha(\chi) = \mat{a}{x}{x \zeta}{a} \in \mathfrak{g'}$ and $Y = \mat{t}{y}{y \zeta}{t} \in \mathfrak{g'}$.  Then $$\gamma(\alpha(\chi), Y) = (x, \zeta) (y, \zeta) \ \gamma_F(\zeta, \psi).$$
\end{lemma}

\proof
First note that the matrices $A = \mat{1}{0}{0}{-1}$ and $B = \mat{0}{1}{-\zeta}{0}$ form a basis for $\mathfrak{g}^{\bot}$.  Let $V = \mat{a}{b}{-b \zeta}{-a}, W = \mat{c}{d}{-d \zeta}{-c} \in \mathfrak{g}^{\bot}$ be arbitrary.  Then $$Q_{(\alpha(\chi),Y)}(V,W) = -4bdxy \zeta^2 + 4acxy \zeta.$$

Now $Q_{(\alpha(\chi),Y)}(A,A)  = 4xy \zeta$, $Q_{(\alpha(\chi),Y)}(B,B) = -4xy \zeta^2$, and $Q_{(\alpha(\chi),Y)}(A,B) = Q_{(\alpha(\chi),Y)}(B,A) = 0$.  Therefore the matrix of the quadratic form $Q_{(\alpha(\chi),Y)}$ is $$\mat{4xy \zeta}{0}{0}{-4xy \zeta^2}.$$

By definition \ref{rao2}, we have $$\gamma(\psi \circ Q_{(\alpha(\chi),Y)}) = h_F(Q_{(\alpha(\chi),Y)})\gamma_F(\psi)^n \gamma_F(\mathrm{det}(Q_{(\alpha(\chi),Y)}),\psi).$$

In our case, $n = 2$, and we also have that by lemma \ref{rao3},

$$h_F(Q_{(\alpha(\chi),Y)}) = (4xy \zeta, -4xy \zeta^2) = (\zeta, -xy).$$  Then, lemma \ref{rao1} implies that

$$\gamma_F(\psi)^2 = \gamma_F(-1, \psi)^{-1} \ \mathrm{and}$$

$$\gamma_F(\mathrm{det}(Q_{(\alpha(\chi),Y)}), \psi) = \gamma_F(-16x^2 y^2 \zeta^3, \psi) = \gamma_F(-\zeta, \psi).$$

Thus, by lemma \ref{rao1}, $$\gamma(\psi \circ Q_{(\alpha(\chi),Y)}) = (\zeta, -xy) \gamma_F(-1, \psi)^{-1} \gamma_F(-\zeta, \psi) = $$ $$(\zeta, -xy) \ \gamma_F(\zeta, \psi) (-1, \zeta)   = (x, \zeta)(y, \zeta) \gamma_F(\zeta, \psi). \qedhere $$

\begin{definition}\label{definitionofgammafactor}
Let $(E/F, \chi)$ be an admissible pair such that $\chi$ has level $m \geq 1$.  Let $\alpha(\chi) \in \mathfrak{p}_E^{-m}$ satisfy $\chi(1+x) = \psi(Tr_{E/F}(\alpha(\chi) x)), x \in \mathfrak{p}_E^m$ (see \cite[Section 19]{bushnellhenniart}).  Associated to $(E/F, \chi)$ is a supercuspidal representation $\pi := \pi_{\chi}$ via theorem \ref{positivedepth}.  Now, $\alpha(\chi) \in E^*$, so $\alpha(\chi) = a + x_{\pi} \delta$ for some $a, x_{\pi} \in F$.  Let $\mathrm{deg}(\pi)$ denote the formal degree of $\pi$. Recall from section \ref{setup} our choice of isomorphism $T(F) \cong E^* / F^*$, giving rise to the standard positive root $z \mapsto z / \overline{z}$, for $z \in E^* / F^*$.  Let $\Delta^+$ be a choice of a positive root of $GL(2,\overline{F})$ with respect to the maximal torus $T(\overline{F})$.  Define $\epsilon(\Delta^+)$ to be 1 if $\Delta^+$ is the standard positive root and define $\epsilon(\Delta^+)$ to be $\tilde\tau(-1)$ if $\Delta^+$ is the opposite root. Then, we define $\epsilon(\tilde\chi, \Delta^+, \tau) := \mathrm{deg}(\pi)(x_{\pi}, \zeta) \gamma_F(\zeta, \psi) \tilde\tau(2 \delta) c_{\psi}(\mathfrak{g'}) c_{\psi}^{-1}(\mathfrak{g})  |\eta(\alpha(\chi))|^{-\frac{1}{2}} \epsilon(\Delta^+)$, where $c_{\psi}(\mathfrak{g'}), c_{\psi}(\mathfrak{g}),$ and  $\eta(\alpha(\chi))$ are defined in \cite[Section 5]{debacker}.
\end{definition}

In the calculations we will make throughout the rest of this section and the next, we will make a choice of $\Delta^+$ to be the standard set of positive roots.  Therefore, the term $\epsilon(\Delta^+)$ is just $1$, and therefore this term will not appear in most of our calculations and formulas.  We have shown in section \ref{setup} that our results are independent of the choice of $\Delta^+$.  When we consider the depth zero setting in section \ref{depthzerochapter}, we will define another constant $\epsilon(\tilde\chi, \Delta^+, \tau)$ in terms of $\epsilon(\Delta^+)$, and we will show in that section that our results are independent of the choice of $\Delta^+$.

\subsection{On certain decompositions associated to elements of $E^*$}\label{decompositions}

We need to understand the sets $n(w) = 0$ and $0 < n(w) \leq r/2$ (cf. theorem \ref{DeBacker1}).  We recall some relevant notions and definitions from \cite[Section 5.3, Section 3.2]{debacker}.  In general, we define a filtration on $E^*$ by setting $G_t' := 1 + \mathfrak{p}_E^{\lceil te \rceil}$ for $t > 0$, where $e$ is the ramification index of $E$ over $F$ and $G' := E^*$.  We also define $G_{t^+}' := \bigcup_{s > t} G_s'$ for $t > 0$.  These are just the Moy-Prasad filtrations from \cite{moyprasad}.  We let $Z(G)$ denote the center of $G = \mathrm{GL}(2,F)$.

\begin{definition}
Let $w \in Z(G) G_{0^+}' \setminus Z(G)$.  Then $n(w)$ is defined by $w \in Z(G) G_{n(w)}' \setminus Z(G) G_{n(w)^+}'$.  A \emph{decomposition of w} by definition is the writing of $w$ in the form $w = ab$ where $a \in Z(G), b \in G_{n(w)}'$.
\end{definition}

\begin{definition}
If $w$ is not in $Z(G) G_{0^+}' = F^*(1 + \mathfrak{p}_E)$, then define $n(w) = 0$.
\end{definition}

We will now determine decompositions of all elements of $E^* \setminus F^*$.  We first deal with the situation where $E/F$ is ramified.  Without loss of generality we write $E = F(\sqrt{p})$.  Note that we are not interested in the cases where $w \in F^*$, since distribution characters are only defined on the regular set.

\begin{lemma}
Let $w = p^n u + p^m v \delta \in E^*$, where $u,v \in \mathfrak{o}_F^*$, and $n,m \in \mathbb{Z}$ such that $n \leq m$.  We can rewrite w as $$w = p^n u \left(1 + p^{m-n + \frac{1}{2}} \frac{v}{u} \right).$$  Thus, $w \in F^* (1 + \mathfrak{p}_E^{2m-2n+1})$.
Moreover, w is not in $F^* (1 + \mathfrak{p}_E^{2m-2n+2})$.  Therefore, a decomposition of w is
$$w = p^n u \left(1 + p^{m-n+\frac{1}{2}} \frac{v}{u} \right) = p^n u \left(1 + \sqrt{p}^{2m-2n+1} \frac{v}{u} \right).$$
\end{lemma}

\begin{proof}
Suppose by way of contradiction that $w = x(1 + s(\sqrt{p})^{2m-2n+2}), s \in \mathfrak{o}_E^*, x \in F^*$.  Write $x^{-1}p^n u(1 + v' p^{m-n+1/2}) = 1 + sp^{m-n+1}$ where $v' = \frac{v}{u}$.   $x \in F^*$ is arbitrary, therefore $x^{-1} p^n u \in F^*$ is arbitrary, so the proof of the lemma reduces to showing that there is no $y \in F^*$ such that $y(1 + v' p^{m-n+1/2}) = 1 + sp^{m-n+1}$.  One may now consider power series expansions of both sides and compare leading coefficients to get the result.
\end{proof}

\begin{lemma}
(1) \ Let $w = p^n u + p^m v \delta \in E^*$ where $u = 0, v \neq 0$. Then $n(w) = 0$.

(2) \ Let $w = p^n u + p^m v \delta \in E^*$, where $u,v \neq 0$, $n > m$. Then $n(w) = 0$.
\end{lemma}

\begin{proof}
The proofs are similar.
\end{proof}

We now describe decompositions of elements of $E^*$ when $E/F$ is unramified.  Then $E = F(\sqrt{\zeta})$ where $\zeta \in F^*$ is a non-square unit.  The proofs of the following lemmas are similar as in the ramified case.

\begin{lemma}
Let $w = p^n u + p^m v \delta \in E^*$, where $u,v \in F$ both non zero, $n < m$.  Then $n(w) > 0$ and a decomposition of w is $$w = p^n u \left(1 + p^{m-n} \frac{\delta v}{u} \right).$$
\end{lemma}

\begin{lemma}
(1) Let $w = p^n u + p^m v \delta \in E^*$, where $u = 0, v \neq 0$.  Then $n(w) = 0$.

(2) Let $w = p^n u + p^m v \delta \in E^*$, where $u,v \neq 0$ and $n \geq m$.  Then $n(w) = 0$.
\end{lemma}

\subsection{On the proof that our character formulas agree with positive depth supercuspidal characters}\label{theproofofcharacters}

Here we prove that on the range $\{ z \in T(F)^{reg} : 0 \leq n(z) \leq r/2 \}$, where $T(F)$ is an elliptic torus in $\mathrm{PGL}(2,F)$, our conjectured character formula agrees with the character of a specific positive depth supercuspidal representation of $\mathrm{PGL}(2,F)$, and this supercuspidal is the one given by the local Langlands correspondence.

In the remainder of the section and the next, we will deal exclusively with $\mathrm{GL}(2,F)$, and so we set $T(F) = E^*$.  Recall from section \ref{setup} that our proposed character formula pulls back to $\mathrm{GL}(2,F)$ to give $$F(\tilde\chi)(w) = \epsilon(\tilde\chi, \Delta^+, \tau) \
\frac{\chi(w) + (-1, \zeta) \chi(\overline{w})}{ \tilde\tau(w-\overline{w}) |D(w)|^{1/2}}, \ \ w \in T(F)^{reg}.$$ It will be useful for computational purposes to rewrite this formula as
$$F(\tilde\chi)(w) = \epsilon(\tilde\chi, \Delta^+, w) \
\frac{\chi(w) + (-1, \zeta) \chi(\overline{w})}{ \tilde\tau(\frac{w-\overline{w}}{2 \delta})}, \ \ w \in T(F)^{reg}.$$
where $\epsilon(\tilde\chi, \Delta^+, w) = \mathrm{deg}(\pi) (x_\pi, \zeta) \gamma_F(\zeta, \psi) C \epsilon(\Delta^+)$ and $$C := c_{\psi}(\mathfrak{g'}) c_{\psi}^{-1}(\mathfrak{g}) |D(w)|^{-1/2} |\eta(\alpha(\chi))|^{-1/2}$$ is as in section \ref{theconstantepsilon}.  We will use this rewritten version for the rest of section 5.  Note that $F(\tilde\chi)$ is independent of the choice of $\tau$ because $\frac{w - \overline{w}}{2 \delta} \in F^*$, and we have required only that $\tilde\tau|_{F^*} = \aleph_{E/F}$.  We will start by assuming that all of our admissible pairs $(E/F, \chi)$ are \emph{minimal} (cf. definition \ref{minimaladmissible}), and then we will show that there is no harm in assuming this, and that all of our results are true for arbitrary admissible pairs.

First we consider the case that $E/F$ is ramified, so we may take $E = F(\sqrt{p})$ without loss of generality.  In other words, we take $\zeta = p$.  The same proofs work in the case $E = F(\sqrt{dp})$, where $d \in \mathfrak{o}_F^*$ is not a square.  We must first conduct a careful analysis of the supercuspidal characters in the $0 < n(w) \leq r/2$ range.  Recall that on the regular set, $n(w) > 0$ if and only if $w = p^n u + p^m v \delta$ where $n \leq m$.  First we need the following lemma.

\begin{lemma}{\label{infinitesquareroot}}
Let F be a local field whose residual characteristic is not 2.  Suppose $\lambda$ is a character of $F^*$ whose order is a power of 2. Then $\lambda|_{1 + \mathfrak{p}_E} \equiv 1$.
\end{lemma}

\begin{proof}
Let $w \in 1 + \mathfrak{p}_E$.  Then $w$ is a square since the leading term in the power series expansion is $1$.  Moreover, one of the square roots of $w$ is in $1 + \mathfrak{p}_E$.  We may then proceed inductively to conclude that $w$ is a $2^n$-th power for any $n \in \mathbb{N}$.
\end{proof}

\begin{lemma}\label{gammafactors}
Let $w \in E^*$ such that $n(w) > 0$, and write a decomposition $w = c(1 + Y)$ as in section \ref{decompositions} and theorem \ref{DeBacker1}.  Then $\gamma(\alpha(\chi), Y) = (x_\pi, \zeta) \gamma_F(\zeta, \psi) \nu(\frac{w - \overline{w}}{2 \delta}) \nu(w)$ for any character $\nu$ of $E^*$ whose restriction to $F^*$ is $\aleph_{E/F}$ and whose order is a power of 2.  Similarly, $\gamma(\alpha(\chi), {}^s Y) = (x_\pi, \zeta) \gamma_F(\zeta, \psi)\nu(\frac{\overline{w} - w}{2 \delta}) \nu(\overline{w}) \ \forall w \in E^* : n(w) > 0$ where $1 \neq s \in W = W(G(F),T(F))$.
\end{lemma}

\begin{proof}
Recall from lemma \ref{weilcalculation} that if $\alpha(\chi), Y$ embed in $\mathfrak{g}'$ as $\alpha(\chi) = \mat{a}{x_{\pi}}{x_{\pi} \zeta}{a}$ and $Y = \mat{t}{y}{y \zeta}{t}$, then $\gamma(\alpha(\chi), Y) = (x_{\pi}, \zeta) (y, \zeta) \ \gamma_F(\zeta, \psi)$.  If $w = m + n \delta \in E^*$ where $m,n \in F^*$, then a decomposition of $w$ is $w = m(1 + (n/m) \delta)$.  Thus, since $(n/m, \zeta) = (n, \zeta) (m, \zeta)$, and since $m = \frac{w+\overline{w}}{2}$ and $n = \frac{w - \overline{w}}{2 \delta}$, we have that $$\gamma(\alpha(\chi), Y) = (x_\pi, \zeta) \gamma_F(\zeta, \psi) \nu \left(\frac{w - \overline{w}}{2 \delta} \right) \nu \left(\frac{w + \overline{w}}{2} \right).$$

\noindent Now, $\nu(\frac{w + \overline{w}}{2}) = \nu(w) \nu(\frac{1}{2}(1 + \overline{w}/w))$, but since $n(w) > 0$, we have $w \in F^* (1 + \mathfrak{p}_E)$, so $\overline{w}/w \in 1 + \mathfrak{p}_E$. So $\frac{1}{2}(1 + \overline{w}/w) \in 1 + \mathfrak{p}_E$, and therefore $\nu(\frac{1}{2}(1 + \overline{w}/w)) = 1$ by lemma \ref{infinitesquareroot}.
\end{proof}

Therefore, we can simplify the supercuspidal character of theorem \ref{DeBacker1}, in the $0 < n(w) \leq r/2$ range, to $$\theta_{\pi}(w) = \epsilon(\tilde\phi, \Delta^+, w)  \ \frac{\phi(w) \nu(w) + (-1, \zeta) \phi(\overline{w}) \nu(\overline{w})}{\nu(\frac{w-\overline{w}}{2 \delta})} \ \forall w \in E^* : 0 < n(w) \leq r/2$$ where $\nu$ is any character of $E^*$ whose restriction to $F^*$ is $\aleph_{E/F}$ and whose order is a power of 2.  Note that $\nu(\frac{w-\overline{w}}{2 \delta}) = \pm 1$, and we have moved this term to the denominator in $\theta_{\pi}(w)$.

Suppose we set $\phi = \chi \nu^{-1}$, where $\nu$ is some character of $E^*$ whose restriction to $F^*$ is $\aleph_{E/F}$ and whose order is a power of 2. It is easy to see that $\alpha(\chi) = \alpha(\phi)$ since $\nu|_{1 + \mathfrak{p}_E} \equiv 1$.  Therefore, $\epsilon(\tilde\phi, \Delta^+, w) = \epsilon(\tilde\chi, \Delta^+, w)$.  Therefore, the above analysis shows $F(\tilde\chi)(w) = \theta_{\pi}(w) \ \forall w \in E^* : 0 < n(w) \leq r/2$ where $\phi = \chi \nu^{-1}$ for any character $\nu$ of $E^*$ whose restriction to $F^*$ is $\aleph_{E/F}$ and whose order is a power of 2.

We will need to investigate the $n(w) = 0$ range to see which such characters $\nu$ can arise, if any.  We will show that our conjectured formula agrees with a supercuspidal character in the $n(w) = 0$ range, for a unique $\nu$.  We will also show that $\nu = {}_F \mu_{\chi}^{-1}$, and so $\phi = \chi {}_F \mu_{\chi}$.

In the next lemma, we will calculate $F(\tilde\chi)$ on the $n(w) = 0$ range, giving a simpler formula in terms of any character $\Omega$ of $E^*$ whose restriction to $F^*$ is $\aleph_{E/F}$ and whose order is a power of 2.  We introduce the new notation $\Omega$ instead of using $\nu$ for the following reason.  We want to show now that, on the $n(w) = 0$ range, $F(\tilde\chi)$ agrees with $\theta_{\pi}$ for $\phi = \chi \nu^{-1}$ for some character $\nu$ of $E^*$ whose restriction to $F^*$ is $\aleph_{E/F}$ and whose order is a power of 2.  We want to show that there is some choice of $\Omega$ and some choice of $\nu$, such that our conjectural character formula $F(\tilde\chi)$ agrees with $\theta_{\pi}$ on the $n(w) = 0$ range, and that these choices are in fact unique.  Both character formulas do not depend on the choice of $\Omega$ and $\nu$, on the $n(w) = 0$ range.  However, they agree on the $n(w) = 0$ range only for a specific choice of $\Omega$ and $\nu$, and this is where the character twist ${}_F \mu_{\chi}$ will pop out.  Indeed, we will show in the end that we are forced to set $\Omega = \nu^{-1}$.  Therefore, if we used the notation $\nu$ instead of $\Omega$, this may have caused serious confusion, since the reader may have thought that $F(\tilde\chi)$ and $\theta_{\pi}$ agreed on the $n(w) = 0$ range if one sets $\phi = \chi \nu^{-1}$ and one sets $\Omega = \nu$ in the next lemma (indeed, this would not hold since $\nu$ is sometimes quartic).  We note that we will use the notation $\Omega$ again when we consider the case that $E/F$ is unramified.

\begin{lemma}
$$F(\tilde\chi)(w) = \epsilon(\tilde\chi, \Delta^+, w) \
(\chi(w)\Omega(\frac{w}{\delta}) + \chi(\overline{w})\Omega(\frac{\overline{w}}{\delta})) \ \ \forall w \in E^* : n(w) = 0$$ for any character $\Omega$ of $E^*$ whose restriction to $F^*$ is $\aleph_{E/F}$ and whose order is a power of 2.
\end{lemma}

\proof
Recall that $n(w) = 0$ if and only if either i) $w = p^n u + p^m v \delta$, where $u,v$ are both nonzero, $n > m$, or ii) $w = p^m v \delta$, where $v$ is nonzero.  In case (i), $\frac{w}{\delta} = p^m v + p^{n} \frac{u}{\delta} = p^m v(1 + p^{n-m} \frac{u}{v \delta})$.  Note that $1 + p^{n-m} \frac{u}{v \delta} \in 1 + \mathfrak{p}_E$.  $\Omega$ is trivial on $1 + \mathfrak{p}_E$ by lemma \ref{infinitesquareroot}, and thus we have that $\Omega(\frac{w}{\delta}) = \Omega(p^m v(1 + p^{n-m} \frac{u}{v \delta})) = \Omega(p^m v) = \Omega(\frac{w - \overline{w}}{2 \delta}) = \tilde\tau(\frac{w - \overline{w}}{2 \delta})$.  In case (ii), it's clear that $\Omega(\frac{w}{\delta}) = \Omega(\frac{w-\overline{w}}{2 \delta}) = \tilde\tau(\frac{w-\overline{w}}{2 \delta})$. Therefore, we get that $$\tilde\tau(\frac{w-\overline{w}}{2 \delta}) = \Omega(\frac{w}{\delta}) \ \forall w \in E^* : n(w) = 0.$$

Thus, since $\tilde\tau(-1) = (-1, \zeta)$, $F(\tilde\chi)$ simplifies in the $n(w) = 0$ range. $$F(\tilde\chi)(w) = \epsilon(\tilde\chi, \Delta^+, w) \
(\chi(w)\Omega(\frac{w}{\delta}) + \chi(\overline{w})\Omega(\frac{\overline{w}}{\delta})) \ \ \forall w \in E^* : n(w) = 0.$$ Note that $\tilde\tau(\frac{w-\overline{w}}{2 \delta}) = \pm 1$, and we have moved this term from the denominator of $F(\tilde\chi)(w)$ to the numerator.
\qed

\begin{lemma}
$F(\tilde\chi)(w) = \theta_{\pi}(w) \ \forall w \in E^* : n(w) = 0$, for $\phi = \chi \nu^{-1}$, for some character $\nu$ of $E^*$ whose restriction to $F^*$ is $\aleph_{E/F}$ and whose order is a power of 2.  In particular, $F(\tilde\chi)(w) = \theta_{\pi}(w) \ \forall w \in E^* : n(w) = 0$ for $\phi = \chi {}_F \mu_{\chi}$.
\end{lemma}

\begin{proof}
Since $E/F$ is ramified and since $\mathfrak g_{x, r/2} \ne \mathfrak g_{x, (r/2)^+}$, $\lambda(\sigma) = 1$ (see \cite{debacker} for relevant notation).  Unwinding the definitions, one can see that to prove the lemma, it suffices to show that $\Omega(\delta) = (x_\pi, \zeta)
\gamma_F(\zeta, \psi) \Omega(w) \nu(w)$.  If we let $\Omega = \nu^{-1}$, then we are reduced to showing that there is some character $\nu$ of $E^*$ whose restriction to $F^*$ is $\aleph_{E/F}$, and whose order is a power of 2, such that $\nu^{-1}(\delta) = (x_\pi, \zeta)
\gamma_F(\zeta, \psi)$.  We claim that $\nu := {}_F \mu_{\chi}^{-1}$ is such a character.

So we want to show that ${}_F \mu_{\chi}(\delta) = (x_\pi,
\zeta) \gamma_F(\zeta, \psi)$.  We need to investigate the term $\alpha(\phi)$.  Note that $\alpha(\phi) = \alpha(\chi)$ since $\nu|_{1 + \mathfrak{p}_E} \equiv 1$. We prefer to work with $\alpha(\chi)$.  Firstly, since $E/F$ is ramified, $\chi$ has odd
level $n = 2m+1$ \cite[Section 19]{bushnellhenniart}.  We also have $\alpha(\chi) \in \mathfrak{p}_E^{-n}$ (cf. section 4.3).  So let $\alpha(\chi) = p^k u + p^{\ell} v
\delta$.  The fact that $n$ is odd clearly implies that $\ell < k$.  Therefore, rewriting $\alpha(\chi)$ as $\alpha(\chi) =
p^{\ell} \sqrt{p}(v + p^{k-\ell-1/2} u)$, we get
$\aleph_{E/F}(\Gamma(\alpha(\chi), \varpi)) = (v_0, \zeta)$ where $v_0$ is the leading term of the power series expansion of $v$. But $(v_0, \zeta) = (v, \zeta)$.  Moreover, $\alpha(\chi) \in \mathfrak{p}_E^{-n}$ implies that $-n = 2 \ell + 1$.  Then, by
definition of $x_\pi$, we get that $x_\pi = p^{\ell} v$.  Now, by definition of ${}_F \mu_{\chi}$, we have ${}_F \mu_{\chi}(\delta) = (v, \zeta) \lambda_{E/F}(\psi)^n$  and  we wish to show that this is equal to
$(x_\pi, \zeta) \gamma_F(\zeta, \psi) = (p^{\ell} v, \zeta) \gamma_F(\zeta,
\psi)$.  Cancelling out terms, we want to show that $(p^{\ell}, \zeta)
\gamma_F(\zeta, \psi) = \lambda_{E/F}(\psi)^{-2 \ell-1}$ since $-n = 2 \ell +
1$.  Well, we know that $$\lambda_{E/F}(\psi)^{2 \ell+2} = (-1,\zeta)^{\ell+1}$$  (cf. \cite[page 217]{bushnellhenniart}).  Multiplying the conjectural equality $(p^{\ell}, \zeta)
\gamma_F(\zeta, \psi) = \lambda_{E/F}(\psi)^{-2 \ell-1}$ by $\lambda_{E/F}(\psi)^{2 \ell + 2}$ on both sides, we are now reduced to showing that $\lambda_{E/F}(\psi) = \gamma_F(\zeta, \psi) (p^{\ell}, \zeta) (-1, \zeta)^{\ell+1}$. But recall that we have assumed from the beginning that $\zeta = p$.  Moreover, $(-1, \zeta) = (\zeta, \zeta)$.  Therefore, $(p^{\ell}, \zeta) (-1, \zeta)^{\ell+1} = (\zeta, \zeta)^{\ell} (-1,\zeta)^{\ell+1} = (-1, \zeta)^{2 \ell + 1} = (-1, \zeta)$.  Therefore, we are reduced to showing that $\lambda_{E/F}(\psi) = \gamma_F(\zeta, \psi) (-1, \zeta)$.  But we proved this in lemma \ref{psilevel}.
\end{proof}

Therefore, we have proven the following, when $E/F$ is ramified.

\begin{theorem}
$F(\tilde\chi)$ agrees with the character of the supercuspidal representation $\pi_{\chi {}_F \mu_{\chi}}$ on the range $\{w \in E^* : 0 \leq n(w) \leq r/2 \}$.
\end{theorem}

What we have actually proven is that if $(E/F, \chi)$ is a \emph{minimal} admissible pair with $E/F$ ramified and $\chi$ having positive level, then $F(\tilde\chi)$ agrees with the character of the supercuspidal representation $\pi_{\chi {}_F \mu_{\chi}}$ on the range $\{w \in E^* : 0 \leq n(w) \leq r/2 \}$.  To prove this for an arbitrary admissible pair follows from this.  For if $(E/F, \chi)$ is an arbitrary admissible pair, then there exists a minimal admissible pair $(E/F, \chi')$ such that $\chi = \chi' \phi_E$ where $\phi_E = \phi \circ N_{E/F}$ for some $\phi \in \widehat{F^*}$.  Moreover, $\pi_{\chi} = \phi \pi_{\chi'}$ by definition.  We proved above that $F(\widetilde{\chi'}) = \theta_{\pi_{\chi' {}_F \mu_{\chi'}}}$ on the range $\{w \in E^* : 0 \leq n(w) \leq r/2 \}$.  Therefore, $\theta_{\pi_{\chi {}_F \mu_{\chi}}}(w) = \theta_{\pi_{\chi' \phi_E {}_F \mu_{\chi' \phi_E}}}(w) = \theta_{\pi_{\chi' \phi_E {}_F \mu_{\chi'}}}(w) = \phi_E(w) \theta_{\pi_{\chi' {}_F \mu_{\chi'}}}(w) = \phi_E(w) F(\widetilde{\chi'})(w) = F(\widetilde{\chi'\phi_E})(w) = F(\tilde\chi)(w)$ on the range $\{w \in E^* : 0 \leq n(w) \leq r/2 \}$.

\

Now we consider the case $E/F$ is unramified, so $\delta = \sqrt{\zeta}$, where $\zeta \in \mathfrak{o}_F^*$ is not a square. Note that $(-1, \zeta) = 1$ since $E/F$ is unramified.

We again first conduct a careful analysis of the supercuspidal characters evaluated on the range $\{w \in E^* : 0 < n(w) \leq r/2 \}$.  Recall that on the regular set, $n(w) > 0$ if and only if $w = p^n u + p^m v \delta, \ u,v \in \mathfrak{o}_F^*$ where $n < m$.

\begin{proposition}
$F(\tilde\chi)$ agrees with the character of the supercuspidal representation $\pi_{\chi \nu^{-1}}$ in the $0 < n(w) \leq r/2$ range, where $\nu$ is any character of $E^*$ whose restriction to $F^*$ is $\aleph_{E/F}$ and whose order is a power of 2.
\end{proposition}

\proof
Again we claim that $$\gamma(\alpha(\chi), Y) = (x_\pi, \zeta) \gamma_F(\zeta, \psi)
\nu(\frac{w - \overline{w}}{2 \delta}) \nu(w)$$  $$\gamma(\alpha(\chi), {}^s Y) = (x_\pi, \zeta) \gamma_F(\zeta, \psi)
\nu(\frac{\overline{w} - w}{2 \delta}) \nu(\overline{w})$$ where $1 \neq s \in W = \mathrm{Aut}(E/F)$ and $\alpha(\chi) = a + x_{\pi} \delta$, for $a, x_{\pi} \in F$.  The reasoning is similar as in lemma \ref{gammafactors}.  One must prove that $\nu(\frac{w + \overline{w}}{2}) = \nu(w) \ \forall w \in E^* : n(w) > 0$.  This is elementary.

Therefore, we can simplify the supercuspidal character of theorem \ref{DeBacker1}, on the $0 < n(w) \leq r/2$ range, to $$\theta_{\pi}(w) = \epsilon(\tilde\phi, \Delta^+, w) \ \frac{\phi(w) \nu(w) + (-1, \zeta) \phi(\overline{w}) \nu(\overline{w})}{\nu(\frac{w-\overline{w}}{2 \delta})} \ \ \forall w \in E^* : 0 < n(w) \leq r/2. \qedhere $$

Suppose we set $\phi = \chi \nu^{-1}$, where $\nu$ is some character of $E^*$ whose restriction to $F^*$ is $\aleph_{E/F}$ and whose order is a power of 2.  Again, as in the ramified case, $\alpha(\chi) = \alpha(\phi)$, since $\nu|_{1 + \mathfrak{p}_E} \equiv 1$.  Therefore, $\epsilon(\tilde\phi, \Delta^+, w) = \epsilon(\tilde\chi, \Delta^+, w)$. Therefore, the above analysis shows $F(\tilde\chi)(w) = \theta_{\pi}(w) \ \forall w \in E^* : 0 < n(w) < r/2$ where $\phi = \chi \nu^{-1}$ for any character $\nu$ of $E^*$ whose restriction to $F^*$ is $\aleph_{E/F}$ and whose order is a power of 2.

We will need to investigate the $n(w) = 0$ range to see which such characters $\nu$ can arise, if any.  We will show that our conjectured formula agrees with a supercuspidal character in the $n(w) = 0$ range, for a unique $\nu$.  We will also show that $\nu = {}_F \mu_{\chi}^{-1}$, and so $\phi = \chi {}_F \mu_{\chi}$.

\begin{lemma}
$$F(\tilde\chi)(w) = \epsilon(\tilde\chi, \Delta^+, w) \
(\chi(w)\Omega(\frac{w}{\delta}) + \chi(\overline{w})\Omega(\frac{\overline{w}}{\delta})) \ \forall w \in E^* : n(w) = 0$$ for the unique unramified character $\Omega$ of $E^*$ whose restriction to $F^*$ is $\aleph_{E/F}$.
\end{lemma}

\proof
Recall that $n(w) = 0$ if and only if either i) $w = p^n u + p^m v \delta$, where $u,v$ are both nonzero, $n > m$, or ii) $w = p^m v \delta$, where $v$ is nonzero, or iii) $w = p^n u + p^m v \delta$, where $n = m$ and $u,v$ are both non-zero.  We first show that $\tilde\tau(\frac{w-\overline{w}}{2 \delta}) = \Omega(\frac{w}{\delta}) \ \forall w \in E^* : n(w) = 0$. In case (i), $\frac{w}{\delta} = p^m v + p^n \frac{u}{\delta} = p^m v(1 + p^{n-m} \frac{u}{v \delta})$.  Note that $1 + p^{n-m} \frac{u}{v \delta} \in 1 + \mathfrak{p}_E$.  But $\Omega$ is trivial on $1 + \mathfrak{p}_E$ since $\Omega$ is unramified.  Therefore, we have $\Omega(\frac{w}{\delta}) = \Omega(p^m v(1 + p^{n-m} \frac{u}{v \delta})) = \Omega(p^m v) = \Omega(\frac{w - \overline{w}}{2 \delta}) = \tilde\tau(\frac{w - \overline{w}}{2 \delta})$.  In case (ii), it's clear that $\Omega(\frac{w}{\delta}) = \Omega(\frac{w-\overline{w}}{2 \delta}) = \tilde\tau(\frac{w-\overline{w}}{2 \delta})$.  In case (iii), $w = p^n (u + v \delta)$.  But $u + v \delta \in \mathfrak{o}_E^*$, and therefore $\Omega(w) = \Omega(p^n)\Omega(u+v \delta) = \Omega(p^n) = \Omega(p^n v)$ since $\Omega$ is unramified.  In particular, since $\Omega(\delta) = 1$, we get $\Omega(\frac{w}{\delta}) = \tilde\tau(\frac{w-\overline{w}}{2 \delta})$.  Therefore, $$\tilde\tau(\frac{w-\overline{w}}{2 \delta}) = \Omega(\frac{w}{\delta}) \ \forall w \in E^* : n(w) = 0.$$

Since $\tilde\tau(-1) = (-1, \zeta) = 1$, $F(\tilde\chi)$ simplifies in the $n(w) = 0$ range. $$F(\tilde\chi)(z) = \epsilon(\tilde\chi, \Delta^+, w) \ (\chi(w)\tilde\tau(\frac{w-\overline{w}}{2
\delta}) + \chi(\overline{w})\tilde\tau(\frac{\overline{w}-w}{2
\delta})) =$$ $$ \epsilon(\tilde\chi, \Delta^+, w) \
(\chi(w)\Omega(\frac{w}{\delta}) + \chi(\overline{w})\Omega(\frac{\overline{w}}{\delta})) \ \forall w \in E^* : n(w) = 0. \qedhere $$

We want to show that $F(\tilde\chi)(w) = \theta_{\pi}(w) \ \forall w \in E^* : n(w) = 0$ for $\phi = \chi \nu^{-1}$ for some character $\nu$ of $E^*$ whose restriction to $F^*$ is $\aleph_{E/F}$ and whose order is a power of 2.  Note that $\Omega(\delta) = 1$ since $\Omega$ is unramified. Unwinding the definitions, one can see that to prove $F(\tilde\chi)(w) = \theta_{\pi}(w) \ \forall w \in E^* : n(w) = 0$ for $\phi = \chi \nu^{-1}$, it suffices to show that $$(x_\pi, \zeta) \gamma_F(\zeta, \psi) \chi(w) \Omega(w) = \chi(w) \nu^{-1}(w) \lambda(\sigma) \ \forall w \in E^* : n(w) = 0.$$

\begin{lemma}
$\lambda(\sigma) = (x_\pi, \zeta) \gamma_F(\zeta, \psi)$.
\end{lemma}

\begin{proof}
Since $E/F$ is unramified, we have $\lambda(\sigma) = (-1)^{r+1}$ where $r$ is the depth of the supercuspidal representation $\pi_{\phi}$ (cf. \cite[Section 5.3]{debacker}).  Note that the depth of $\pi_{\phi}$ equals the depth of $\pi_{\chi}$ since $\nu|_{1 + \mathfrak{p}_E} \equiv 1$, i.e. $\nu$ has level zero.

Now, consider the term $(x_\pi, \zeta)$.  We need to investigate the term $\alpha(\phi)$.  Note that $\alpha(\phi) = \alpha(\chi)$ since again, $\nu|_{1 + \mathfrak{p}_E} \equiv 1$. We prefer to work with $\alpha(\chi)$.  Recall that $\alpha(\chi) \in \mathfrak{p}_E^{-n}$.  Moreover, since $\alpha(\chi) \in \mathfrak{g}_{-r}' \setminus \mathfrak{g}_{-r^+}'$ (cf. \cite[page 34]{debacker}), we have that $n = r$.  Now let $\alpha(\chi) = p^k u + p^l v \delta$.  We need a lemma.

\begin{lemma}
$l = -n$.
\end{lemma}

\begin{proof}
Since $(E/F, \chi)$ is a minimal admissible pair (cf. \cite[Section 19.2 line 1]{bushnellhenniart}), we have that $\alpha(\chi)$ is a \emph{minimal element} over $F$ (see \cite[Proposition 18.2]{bushnellhenniart}).  By \cite[Section 13.4]{bushnellhenniart}, $\alpha(\chi)$ is minimal over $F$ if and only if $(\alpha(\chi) + \mathfrak{p}_E^{-n+1}) \cap F = \emptyset$, where $n = -v_E(\alpha(\chi))$.  Now, it must be the case that either $l = -n$ or $k = -n$ (or both), since $v_E(\alpha(\chi)) = -n$.  So suppose by way of contradiction that $k = -n$ but $l \neq -n$.  Let $w = p^a u' + p^b v' \delta \in \mathfrak{p}_E^{-n+1}$.  Then $\alpha(\chi) + w \in F$ if and only if $p^l v \delta + p^b v' \delta = 0$.  Then since $\alpha(\chi) \in \mathfrak{p}_E^{-n} \setminus \mathfrak{p}_E^{-n+1}$ and since $k = -n$, we must have that $l \geq -n$.  But we have assumed $l \neq -n$, so that means $l \geq -n+1$.  Thus, $b \geq -n+1$.  Therefore, just pick any $a \in \mathbb{Z}$ such that $a \geq -n+1$ and any $u'$ and set $p^b v' = -p^l v$. Then we get that $p^a u' + p^b v' \delta \in \mathfrak{p}_E^{-n+1}$ and that $\alpha(\chi) + w \in F$.  Thus, $(\alpha(\chi) + \mathfrak{p}_E^{-n+1}) \cap F \neq \emptyset$, a contradiction, and the lemma is proven.
\end{proof}

Returning to the proof of the proposition, recall that $x_\pi = p^l v$.  Thus, $(x_\pi, \zeta) = (p^{-n} v, \zeta)$.  Also, since $r = n$, we have $\lambda(\sigma) = (-1)^{r+1} = (-1)^{n+1}$.  Therefore, we are reduced to showing that

$$(p^{-n} v, \zeta) \gamma_F(\zeta, \psi) = (-1)^{n+1}$$ so equivalently, $\gamma_F(\zeta, \psi) = -1$.  But since $\psi$ has level 1, it is a fact that $\gamma_F(\zeta, \psi) = -1$ by lemma \ref{sponge}.
\end{proof}

\begin{proposition}
$F(\tilde\chi)(w) = \theta_{\pi}(w) \ \forall w \in E^* : n(w) = 0$, for $\phi = \chi \nu^{-1}$, for some character $\nu$ of $E^*$ whose restriction to $F^*$ is $\aleph_{E/F}$ and whose order is a power of 2.  In particular, $F(\tilde\chi)(w) = \theta_{\pi}(w) \ \forall w \in E^* : n(w) = 0$ for $\phi = \chi {}_F \mu_{\chi}$.
\end{proposition}

\begin{proof}
By the previous lemma, the conjectured equation $$(x_\pi, \zeta) \gamma_F(\zeta, \psi) \chi(w) \Omega(w) = \chi(w) \nu^{-1}(w) \lambda(\sigma) \ \forall w \in E^*:  n(w) = 0$$ simplifies to $\Omega(w) = \nu^{-1}(w) \ \forall w \in E^* : n(w) = 0$.  Therefore, $\nu^{-1}$ is an unramified character of $E^*$ whose restriction to $F^*$ is $\aleph_{E/F}$.  This implies that $\nu^{-1} = {}_F \mu_{\chi} = {}_F \mu_{\chi}^{-1}$.
\end{proof}

Therefore, we have proven the following, when $E/F$ is unramified.

\begin{theorem}\label{charactersmatchingup}
$F(\tilde\chi)$ agrees with the character of the supercuspidal representation $\pi_{\chi {}_F \mu_{\chi}}$ on the range $\{w \in E^* : 0 \leq n(w) \leq r/2 \}$.
\end{theorem}

What we have actually proven is that if $(E/F, \chi)$ is a \emph{minimal} admissible pair with $E/F$ unramified and $\chi$ having positive level, then $F(\tilde\chi)$ agrees with the character of the supercuspidal representation $\pi_{\chi {}_F \mu_{\chi}}$ on the range $\{w \in E^* : 0 \leq n(w) \leq r/2 \}$.  To prove this for an arbitrary admissible pair follows.  For if $(E/F, \chi)$ is an arbitrary admissible pair, then there exists a minimal admissible pair $(E/F, \chi')$ such that $\chi = \chi' \phi_E$ where $\phi_E = \phi \circ N_{E/F}$ for some $\phi \in \widehat{F^*}$.  Moreover, $\pi_{\chi} = \phi \pi_{\chi'}$ by definition.  We proved above that $F(\widetilde{\chi'}) = \theta_{\pi_{\chi' {}_F \mu_{\chi'}}}$ on the range $\{w \in E^* : 0 \leq n(w) \leq r/2 \}$.  Therefore, $\theta_{\pi_{\chi {}_F \mu_{\chi}}}(w) = \theta_{\pi_{\chi' \phi_E {}_F \mu_{\chi' \phi_E}}}(w) = \theta_{\pi_{\chi' \phi_E {}_F \mu_{\chi'}}}(w) = \phi_E(w) \theta_{\pi_{\chi' {}_F \mu_{\chi'}}}(w) = \phi_E(w) F(\widetilde{\chi'})(w) = F(\widetilde{\chi'\phi_E})(w) = F(\tilde\chi)(w)$ on the range $\{w \in E^* : 0 \leq n(w) \leq r/2 \}$.

\subsection{On whether there are two positive depth character formulas coming from the same Cartan}\label{samecartan}

In the next two sections we show that a positive depth supercuspidal representation of $\mathrm{GL}(2,F)$ (and hence $\mathrm{PGL}(2,F)$) is uniquely determined by the restriction of its distribution character to the $n(w) = 0$ range.  In this section, we show that if the distribution characters of two positive depth supercuspidal representations, both coming from the same Cartan, agree on the $n(w) = 0$ range, then the supercuspidal representations are isomorphic.  Recall that we have proven in section \ref{theproofofcharacters} that if $(E/F, \chi)$ is an admissible pair, then $$\theta_{\pi_{\chi {}_F \mu_{\chi}}}(w) = \epsilon(\tilde\chi, \Delta^+, w) \
\frac{\chi(w) + (-1, \zeta) \chi(\overline{w})}{ \tilde\tau(\frac{w-\overline{w}}{2 \delta})}$$ on the range $0 \leq n(w) \leq r/2$.  We will use this formula regularly in what follows, as well as in the next section.  We note that all of the supercuspidal representations of $\mathrm{GL}(2,F)$ are given by $\pi_{\chi {}_F \mu_{\chi}}$ for some admissible pair $(E/F, \chi)$ (by theorem \ref{tamellc}).

\begin{theorem}\label{samecartans}
Suppose $(E/F, \chi_1)$ and $(E/F,\chi_2)$ are admissible pairs such that $\theta_{\pi_{\chi_1 {}_F \mu_{\chi_1}}}(w) = \theta_{\pi_{\chi_2 {}_F \mu_{\chi_2}}}(w) \ \forall w \in \{z \in E^* : n(z) = 0 \}$.  Then, $\chi_1 = \chi_2^{\upsilon}$ for some $\upsilon \in \mathrm{Aut}(E/F)$.
\end{theorem}

The proof consists of the following two lemmas.

\begin{lemma}\label{samecartanramifiedgl2}
Let $E/F$ be ramified.  Suppose $(E/F, \chi_1)$ and $(E/F, \chi_2)$ are admissible pairs such that $\theta_{\pi_{\chi_1 {}_F \mu_{\chi_1}}}(w) = \theta_{\pi_{\chi_2 {}_F \mu_{\chi_2}}}(w) \ \forall w \in \{z \in E^* : n(z) = 0 \}$.  Then, $\chi_1 = \chi_2^{\upsilon}$ for some $\upsilon \in \mathrm{Aut}(E/F)$.
\end{lemma}

\begin{proof}
We may assume without loss of generality that $E = F(\sqrt{p})$.  We have assumed that $$\epsilon(\tilde\chi_1, \Delta^+, w) \frac{\chi_1(w) + (-1, \zeta) \chi_1(\overline{w})}{\tilde\tau(\frac{w-\overline{w}}{2 \delta})} =$$ $$ \epsilon(\tilde\chi_2, \Delta^+, w) \frac{\chi_2(w) + (-1, \zeta) \chi_2(\overline{w})}{\tilde\tau(\frac{w-\overline{w}}{2 \delta})} \ \ \forall w \in E^* : n(w) = 0.$$ Let us write $\alpha(\chi)_i$ for the $\alpha(\chi)$ that are associated to the pairs $(E/F, \chi_i)$, and similarly for $x_{\pi_i}$ (see definition \ref{definitionofgammafactor}), and similarly for $\mathrm{deg}(\pi_i)$. Now, let $c_i = \mathrm{deg}(\pi_i)(x_{\pi_i}, \zeta) |\eta(\alpha(\chi)_i)|^{-1/2}$.  Then, cancelling out like terms, we have that $$c_1 (\chi_1(w) + (-1, \zeta) \chi_1(\overline{w})) = c_2 (\chi_2(w) + (-1, \zeta) \chi_2(\overline{w})) \ \forall w \in E^* : n(w) = 0.$$  We will prove in proposition \ref{case1} that there exists a $w' \in E^* \setminus F^*(1 + \mathfrak{p}_E)$ such that $\chi_1(w') + (-1, \zeta) \chi_1(\overline{w'}) \neq 0$.  Therefore, the same proof of \cite[Lemma 5.1]{spice} shows that $\chi_1|_{F^*(1 + \mathfrak{p}_E)} = \chi_2^{\upsilon}|_{F^*(1 + \mathfrak{p}_E)}$ for some $\upsilon \in \mathrm{Aut}(E/F)$.  For the following arguments, it suffices without loss of generality to assume $\upsilon = 1$.

Let $c := \frac{c_1}{c_2}$.  Let $[\chi](w) := \chi(w) + (-1, \zeta) \chi(\overline{w})$.  Then we have $c [\chi_1](w) =  [\chi_2](w) \ \forall w \in E^* \setminus F^*(1 + \mathfrak{p}_E)$ and we also have that $\chi_1|_{F^*(1 + \mathfrak{p}_E)} = \chi_2|_{F^*(1 + \mathfrak{p}_E)}$.  Now let $p_E$ be a uniformizer of $E$ and recall that $p$ is a uniformizer of $F$.  We may take $p_E$ so that $p_E^{2} = p$. Since $\chi_1(p) = \chi_2(p)$, we have that $\chi_1(p_E)^{2} = \chi_2(p_E)^{2}$, and so $\chi_2(p_E) = \xi_{2} \chi_1(p_E)$ where $\xi_2$ could be plus or minus 1.  Therefore, $\chi_2(w) = \chi_1(w) \xi_{2}^{v_E(w)} \ \forall w \in E^*$.  Therefore, after substituting and noting that $val(w) = val(\upsilon(w))$, we obtain $$c [\chi_1](w) = \xi_{2}^{val(w)} [\chi_1](w) \ \forall w \in E^* \setminus F^*(1 + \mathfrak{p}_E).$$  Again, we will prove in the next section that there exists a $w' \in E^* \setminus F^*(1 + \mathfrak{p}_E)$ such that $[\chi_1](w') \neq 0$. Therefore, we can cancel $[\chi_1](w')$ from both sides to obtain $$c = \xi_{2}^{val(w')}.$$  Therefore, $c$ is plus or minus $1$.

Suppose that $c = 1$.  Then, we get $[\chi_1](w) =  [\chi_2](w) \ \forall w \in E^*$. By linear independence of characters, $\chi_1 = \chi_2$.

Suppose $c = -1$.  Then $\chi_2(w) = \chi_1(w) (-1)^{val(w)}$. This implies that $\chi_2 = \chi_1 \otimes \phi_E$, where $\phi_E := \phi \circ N_{E/F}$ where $\phi = \aleph_{L/F}$ where $L/F$ is the unique unramified degree 2 extension of $F$.  In this case, one can check that $\alpha(\chi)_1 = \alpha(\chi)_2$.  Therefore, $c = \frac{\mathrm{deg}(\pi_1)}{\mathrm{deg}(\pi_2)}$.  But formal degrees are positive real numbers, and so we get a contradiction to the supposition that $c = -1$.

Therefore, $\chi_1 = \chi_2$ or $\chi_1 = \chi_2^{\upsilon}$, and so the admissible pairs are isomorphic.
\end{proof}

\begin{lemma}
Let $E/F$ be unramified.  Suppose $(E/F, \chi_1)$ and $(E/F, \chi_2)$ are admissible pairs such that $\theta_{\pi_{\chi_1 {}_F \mu_{\chi_1}}}(w) = \theta_{\pi_{\chi_2 {}_F \mu_{\chi_2}}}(w) \ \forall w \in \{z \in E^* : n(z) = 0 \}$.  Then, $\chi_1 = \chi_2^{\upsilon}$ for some $\upsilon \in \mathrm{Aut}(E/F)$.
\end{lemma}

\begin{proof}
An analogous argument as in \cite[pages 73-74]{spice} works here.
\end{proof}

\subsection{On whether there are two positive depth character formulas coming from different Cartans}\label{othercartan}

In this section we show that the distribution characters of two positive depth supercuspidal representations, coming from different Cartans, can't agree on the $n(w) = 0$ range.  This, together with the results from the previous section, shows that if $(E/F, \chi)$ is an admissible pair, then there is a unique positive depth supercuspidal representation whose character agrees with $F(\tilde\chi)$ on the range $\{ w \in E^* : n(w) = 0$ \}.

As part of the proof of the following theorem, we will need to show that a supercuspidal character of $\mathrm{GL}(2,F)$ cannot vanish on all of the $n(w) = 0$ elements of its associated Cartan subgroup.  That is, if $(E/F, \chi)$ is an admissible pair, then there exists a $w \in E^*$ such that $n(w) = 0$ and  $\theta_{\pi_{\chi}}(w) \neq 0$.  Again, we will regularly use the formula $$\theta_{\pi_{\chi {}_F \mu_{\chi}}}(w) = \epsilon(\tilde\chi, \Delta^+, w) \
\frac{\chi(w) + (-1, \zeta) \chi(\overline{w})}{ \tilde\tau(\frac{w-\overline{w}}{2 \delta})}$$ on the range $0 \leq n(w) \leq r/2$.

\begin{theorem}\label{differentcartans}
Suppose $(E/F, \chi)$ and $(E_1/F, \chi_1)$ are admissible pairs with $E \ncong E_1$.  Then $\exists w \in E^* : n(w) = 0$ such that $\theta_{\pi_{\chi {}_F \mu_{\chi}}}(w) \neq \theta_{\pi_{\chi_1 {}_F \mu_{\chi_1}}}(w)$.
\end{theorem}

Let $E = F(\sqrt{\zeta_E})$ and $E_1 = F(\sqrt{\zeta_{E_1}})$.  There are many cases to check, and we split them up in a sequence of propositions.

\begin{proposition}\label{case1}
Suppose $(E/F, \chi)$ and $(E_1/F, \chi_1)$ are admissible pairs with $E$ ramified and $E_1$ unramified.  Then $\exists w \in E^* : n(w) = 0$ such that $\theta_{\pi_{\chi {}_F \mu_{\chi}}}(w) \neq \theta_{\pi_{\chi_1 {}_F \mu_{\chi_1}}}(w)$.
\end{proposition}

\begin{proof}
By comparing valuations of determinants of elements, one can show that since the inducing representation of the representation coming from $E_1$ is $E_1^* G_{x_1,r_1/2}$ (where the point $x_1$ comes from $E_1$ and the depth $r_1$ comes from $\chi_1$.  See \cite[page 35]{debacker}), then if $w \in E^* : n(w) = 0$, then one can't conjugate $w$ into $E_1^* G_{x_1,r_1/2}$ .  Therefore, $\theta_{\pi_{\chi_1 {}_F \mu_{\chi_1}}}(w) = 0$.

Suppose by way of contradiction that $\theta_{\pi_{\chi {}_F \mu_{\chi}}}(w) = 0 \ \forall w \in E^* : n(w) = 0$.  Then $\chi(w) + (-1, \zeta_E) \chi(\overline{w}) = 0 \ \forall w \in E^* : n(w) = 0.$  Thus, $\chi(w/\overline{w})^2 = 1 \ \forall w \in E^* : n(w) = 0$.  But since the set $\{w \in E^* : n(w) = 0 \}$ generates all of $E^*$ as a group, we get that $\chi|_{(E^*)^2} = \chi^{\upsilon}|_{(E^*)^2}$.  Now since $1 + \mathfrak{p}_E \subset (E^*)^2$, we get $\chi|_{1 + \mathfrak{p}_E} = \chi^{\upsilon}|_{1 + \mathfrak{p}_E}$ , which contradicts the fact that $(E/F, \chi)$ is an admissible pair.
\end{proof}

\begin{proposition}\label{case2}
Suppose $(E/F, \chi)$ and $(E_1/F, \chi_1)$ are admissible pairs with $E \ncong E_1$, with $E$ unramified and $E_1$ ramified.  Then $\exists w \in E^* : n(w) = 0$ such that $\theta_{\pi_{\chi {}_F \mu_{\chi}}}(w) \neq \theta_{\pi_{\chi_1 {}_F \mu_{\chi_1}}}(w)$.
\end{proposition}

We will need to split this proposition into two cases: $(-1,p) = 1$ and $(-1,p) = -1$.  We have $E = F(\sqrt{\zeta})$, where $\zeta \in \mathfrak{o}_F^*$ is not a square, and without loss of generality $E_1 = F(\sqrt{p})$.

\begin{lemma}\label{subcasea''}
Suppose $(-1,p) = 1$.  Suppose $(E/F, \chi)$ and $(E_1/F, \chi_1)$ are admissible pairs with $E$ unramified and $E_1$ ramified.  Then $\exists w \in E^* : n(w) = 0$ such that $\theta_{\pi_{\chi {}_F \mu_{\chi}}}(w) \neq \theta_{\pi_{\chi_1 {}_F \mu_{\chi_1}}}(w)$.
\end{lemma}

\proof
We use the same strategy as in proposition \ref{case1}.  Recall that $w \in E^* : n(w) = 0$ if and only if
(i) $w = p^m v \delta$,
(ii) $w = p^n u + p^m v \delta, \ n = m$, or
(iii) $w = p^n u + p^m v \delta, \ n > m$.

By comparing determinants again, one can show that if $w$ is of the form (i) or (iii), then $w$ can't be conjugated into $E_1^* G_{x_1,r_1/2}$ (this is easy to show when $(-1,p) = 1$.  We will say something else when we consider the case $(-1,p) = -1$). Therefore, $\theta_{\pi_{\chi_1 {}_F \mu_{\chi_1}}}(w) = 0$ for all $w$ of the form (i) and (iii).  It is also the case that $\theta_{\pi_{\chi_1 {}_F \mu_{\chi_1}}}(w) = 0$ for all $w$ of the form (ii), by \cite[Proposition 2]{shimizu}.  Therefore, $\theta_{\pi_{\chi_1 {}_F \mu_{\chi_1}}}(w) = 0$ for all $w$ such that $n(w) = 0$.

Assume that $\theta_{\pi_{\chi {}_F \mu_{\chi}}}(w) = 0$ for all $w$ such that $n(w) = 0$.  We will now find a contradiction.  We need the following lemma.

\begin{lemma}\label{randomlemma1}
$\chi|_{F^* (1 + \mathfrak{p}_E)} = \chi^{\upsilon}|_{F^* (1 + \mathfrak{p}_E)}$, where $\upsilon$ generates $\mathrm{Aut}(E/F)$.
\end{lemma}

\proof
Let $z \in 1 + \mathfrak{p}_E$.  It is easy to see that one can write $z = w_1 \overline{w_2}$, where $w_1$ is of the form (i) and $w_2$ is of the form (iii).  Now, since $\theta_{\pi_{\chi {}_F \mu_{\chi}}}(w) = 0$ on all $w$ of the form (i) and (iii), we have $$\chi(w_1) + \chi(\overline{w_1}) = 0$$ $$\chi(w_2) + \chi(\overline{w_2}) = 0.$$  Multiplying the first equation by $\chi(w_2)$ and the second equation by $\chi(w_1)$, we conclude that $$\chi(w_2 \overline{w_1}) = \chi(w_1 \overline{w_2}).$$  Therefore, $\chi(z) = \chi(\overline{z}) \ \forall z \in 1 + \mathfrak{p}_E$.  Since $\chi(x) = \chi(\overline{x}) \ \forall x \in F^*$, we get $$\chi|_{F^* (1 + \mathfrak{p}_E)} = \chi^{\upsilon}|_{F^* (1 + \mathfrak{p}_E)}. \qedhere $$

By a similar argument as in the proof of lemma \ref{randomlemma1}, we have that since $\theta_{\pi_{\chi {}_F \mu_{\chi}}}(w) = 0 \ \forall w \in E^* : n(w) = 0$, then $\chi(w_1 \overline{w_2}) = \chi(\overline{w_1} w_2) \ \forall w_1, w_2 \in E^* : n(w_1) = n(w_2) = 0$.  Now, if $w_1$ is of the form (i) and $w_2$ is of the form (ii), then $w_1 \overline{w_2}$ is of the form (ii).  This shows that $\chi(z) = \chi(\overline{z})$ for some element $z$ of the form (ii).  Therefore, $\chi(z) + \chi(\overline{z}) = 2 \chi(z) \neq 0$, which says that $\theta_{\pi_{\chi {}_F \mu_{\chi}}}(z) \neq 0$.  But since $n(z) = 0$, we have contradicted our assumption that $\theta_{\pi_{\chi {}_F \mu_{\chi}}}(w) = 0$ for all $w$ such that $n(w) = 0$.  Thus, we have shown that there must exist an element $w \in E^*$ such that $n(w) = 0$ and $\theta_{\pi_{\chi {}_F \mu_{\chi}}}(w) \neq 0$.  Finally, we are done with proving lemma \ref{subcasea''}.
\qed

\begin{lemma}\label{subcaseb''}
Suppose $(-1,p) = -1$.  Suppose $(E/F, \chi)$ and $(E_1/F, \chi_1)$ are admissible pairs with $E$ unramified and $E_1$ ramified.  Then $\exists w \in E^* : n(w) = 0$ such that $\theta_{\pi_{\chi {}_F \mu_{\chi}}}(w) \neq \theta_{\pi_{\chi_1 {}_F \mu_{\chi_1}}}(w)$.
\end{lemma}

\begin{proof}

Recall again that $w \in E^* : n(w) = 0$ if and only if either (i) $w = p^m v \delta$, (ii) $w = p^n u + p^m v \delta, \ n = m$, or (iii) $w = p^n u + p^m v \delta, \ n > m$.

We want to show like in lemma \ref{subcasea''} that $\theta_{\pi_{\chi_1 {}_F \mu_{\chi_1}}}$ vanishes on elements of the form (i), (ii), and (iii).  After showing this, the rest of the proof of lemma \ref{subcaseb''} goes exactly the same way as in lemma \ref{subcasea''}.  But \cite[Proposition 2]{shimizu} shows that $\theta_{\pi_{\chi_1 {}_F \mu_{\chi_1}}}(w) = 0$ for all $w \in E^*$ such that $w$ is of the form (i),(ii), or (iii).
\end{proof}

Therefore, we have finished the proof of proposition \ref{case2}.

\begin{proposition}\label{case3}
Suppose $(E/F, \chi)$ and $(E_1/F, \chi_1)$ are admissible pairs with $E \ncong E_1$, with $E$ ramified and $E_1$ ramified.  Then $\exists w \in E^* : n(w) = 0$ such that $\theta_{\pi_{\chi {}_F \mu_{\chi}}}(w) \neq \theta_{\pi_{\chi_1 {}_F \mu_{\chi_1}}}(w)$.
\end{proposition}

\begin{proof}
Suppose without loss of generality that $E = F(\sqrt{p})$ and $E_1 = F(\sqrt{\zeta p})$, $\zeta \in \mathfrak{o}_F^*$ not a square.  We first claim that elements $w \in E^*$ such that $n(w) = 0$ can't be conjugated into $E_1^* G_{x_1,r_1/2}$.

Recall that $w \in E^*$ such that $n(w) = 0$ if and only if either $w = p^n u + p^m v \sqrt{p}$ with $n > m$, $u,v \neq 0$, or $w = p^m v \sqrt{p}$ with $v \neq 0$.  Suppose $w = p^m v \sqrt{p}$. Then $\mathrm{det}(w) = N(w) = - p^{2m+1} v^2$.  Note that $\mathrm{det}(G_{x_1,r_1/2}) = 1 + \mathfrak{p}_F$ and $\mathrm{det}(E_1^* G_{x_1, r_1/2}) = N_{E_1 / F}(E_1^*)$.  Well, $(\mathrm{det}(w), \zeta p) = (-p^{2m+1}v^2, \zeta p) = (-p, \zeta p) = (-1, \zeta p) (p, \zeta p) = (\zeta p, \zeta p) (p, \zeta p) = (\zeta p^2, \zeta p) = (\zeta, \zeta p) = (\zeta, \zeta) (\zeta, p) = (\zeta, p) = -1$.  Therefore, $\mathrm{det}(w)$ can't be a norm from $E_1$.
Suppose now that $w = p^n u + p^m v \sqrt{p}$ with $n > m,$ with $u,v \neq 0$.  Then $w = p^m v \sqrt{p} (1+ z)$, for some $z \in \mathfrak{p}_E$.  Therefore, $\mathrm{det}(w) = N(w) = -p^{2m+1} v^2 (1+ z')$ for some $z' \in \mathfrak{p}_F$, and as above we have $(\mathrm{det}(w), \zeta p) = -1$, so $\mathrm{det}(w)$ can't be a norm from $E_1$.  The rest of the proof follows as in proposition \ref{case1} above.
\end{proof}

We have now finished the proof of theorem \ref{differentcartans}.  Summing up, we have altogether shown that if $(E/F, \chi)$ is an admissible pair such that $\chi$ has positive level, then there is a unique positive depth supercuspidal representation, $\pi_{\chi {}_F \mu_{\chi}}$, whose character, on the range $\{ z \in T(F)^{reg} : 0 \leq n(z) \leq r/2 \}$, agrees with $F(\tilde\chi)$.  There is one minor point here to resolve.  Is there possibly a depth zero supercuspidal representation whose character, on the range $\{ z \in T(F)^{reg} : 0 \leq n(z) \leq r/2 \}$, also equals $F(\tilde\chi)$?  We will prove in the next section that if $(E_1/F, \chi_1)$ is an admissible pair corresponding to a depth zero supercuspidal representation $\pi$, then its character formula, on the range $\{ z \in T(F)^{reg} : 0 \leq n(z) \leq r/2 \}$, is $$F(\tilde\chi_1)(w) = - \epsilon(\Delta^+) \frac{\mathrm{deg}(\pi)}{\mathrm{deg}(\sigma)} \left(\frac{\chi_1(w) + (-1, \zeta) \chi_1(\overline{w})}{ \tilde\tau(\frac{w - \overline{w}}{2 \delta})}\right), \ \ w \in E_1^* \setminus F^*(1 + \mathfrak{p}_{E_1}).$$ Here, we are writing $\pi := \mathrm{Ind}_{ZK}^G \sigma$, where $Z$ is the center of $\mathrm{GL}(2,F)$ and $K = GL(2, \mathfrak{o}_F)$, for some representation $\sigma$.  Then, the same arguments as in theorems \ref{samecartans} and \ref{differentcartans} show that the character of $\pi$ cannot equal $F(\tilde\chi)$, on the range $\{ z \in T(F)^{reg} : 0 \leq n(z) \leq r/2 \}$, unless $E = E_1$ and $\chi = \chi_1^{\upsilon}$ for some $\upsilon \in \mathrm{Aut}(E/F)$.  But even this is not possible since by assumption, $\chi_1$ has level zero, whereas $\chi$ has positive level.

Therefore, combining theorems \ref{samecartans}, \ref{differentcartans}, and \ref{charactersmatchingup}, we obtain the following result.

\begin{theorem}
The assignment

\begin{eqnarray}
\{ \mathrm{supercuspidal} \ \phi : W_F \rightarrow SL(2,\mathbb{C}) \} & \mapsto & \tilde\chi \in \widehat{T(F)}_{\tau \circ \rho} \mapsto \pi(\tilde\chi) \nonumber
\end{eqnarray}

\noindent from section \ref{setup} is the Local Langlands correspondence for positive depth supercuspidal representations of $\mathrm{PGL}(2,F)$, where $\pi(\tilde\chi)$ is the unique supercuspidal representation whose character, on the range $\{ z \in T(F)^{reg} : 0 \leq n(z) \leq r/2 \}$, is $F(\tilde\chi)$.
\end{theorem}

\section{Depth zero supercuspidal character formulas for $\mathrm{PGL}(2,F)$}\label{depthzerochapter}

\subsection{On the proof that our conjectural character formulas agree with depth zero supercuspidal characters}

In the following two sections, we prove theorems \ref{gl2theorem1} and \ref{gl2theorem2} for the case of depth zero supercuspidal representations of $\mathrm{PGL}(2,F)$.

Let us recall from the previous section that the proposed character formula simplifies to $$F(\tilde\chi)(w) = \epsilon(\tilde\chi, \Delta^+, \tau)
 \left(\frac{\chi(w) + (-1, \zeta) \chi(\overline{w})}{ \tilde\tau(w - \overline{w}) |D(w)|^{1/2}}\right), \ w \in T(F)^{reg}.$$
For depth zero representations we define $\epsilon(\tilde\chi, \Delta^+, \tau) := \frac{\mathrm{deg}(\pi) \tilde\tau(2 \delta)}{\mathrm{deg}(\sigma)} \epsilon(\Delta^+)$, where $\epsilon(\Delta^+)$ is as in section \ref{theconstantepsilon}.  We will show later that our results are independent of the choice of $\Delta^+$.  Recall from section \ref{setup} our choice of isomorphism $T(F) \cong E^* / F^*$, giving rise to the standard positive root $z \mapsto z / \overline{z}$, for $z \in E^* / F^*$.  We will again here choose $\Delta^+$ to be the standard positive root of $GL(2,\overline{F})$ with respect to the maximal torus $T(\overline{F})$.  Let us recall the following theorem on characters of depth zero supercuspidal representations $\pi := \mathrm{Ind}_{F^* K}^G \sigma$ of $\mathrm{GL}(2,F)$, where we are identifying $F^*$ with the center of $\mathrm{GL}(2,F)$ and $K = \mathrm{GL}(2, \mathfrak{o}_F)$.  We cite the theorem from \cite{debacker}, however the result is due to Phil Kutzko (see \cite{kutzko}).  We only write down the part of the character formula that we need.

\begin{theorem}\label{depthzerocharacters}\footnote{Notice that this theorem is slightly different than the one from \cite{debacker}.  It is because there are a few typos in \cite{debacker}.}
\cite[Theorem 5.4.1]{debacker}

Suppose $\gamma \in F^* K_0^{reg}$.  Then
\begin{equation*}
\frac{\theta_{\pi}(\gamma)}{\mathrm{deg}(\pi)} =
\chi_{\pi}(c) \frac{\chi_{\sigma}(\gamma)}{\mathrm{deg}(\sigma)} \ \text{if } \gamma = cw \ is \ unramified \ elliptic \ and \ \gamma \ \notin F^* K_1, c \in F^*, w \in K_0.
\end{equation*}

\end{theorem}

We now compute the supercuspidal characters of theorem \ref{depthzerocharacters}.  Let $\epsilon_{\mathbb{G}}$ be $(-1)^{r}$, where $r$ is the $\mathbb{F}_q$-rank of $\mathbb{G}$, for any algebraic group $\mathbb{G}$ defined over $\mathbb{F}_q$.  Let $(E/F, \chi)$ be an admissible pair corresponding to a depth zero supercuspidal representation via theorem \ref{positivedepth}.
This means that $E/F$ is unramified and $\chi$ has level zero, so $\chi|_{\mathfrak{o}_E^*}$ gives rise to a character $\theta$ of the multiplicative group
of the residue field $\mathbb{F}_{q^2}$ of $E$.  Now let $\mathbb{G} := \mathrm{GL}(2, \overline{\mathbb{F}_q})$.  Let
$\mathbb{T}$ be the maximal torus of $\mathbb{G}$ defined over $\mathbb{F}_q$ such that $\mathbb{T}^{\Phi} = \mathbb{F}_{q^2}^*$ is the elliptic
torus in $\mathrm{GL}(2, \mathbb{F}_q)$, where $\Phi$ denotes Frobenius.  Then, by Deligne-Lusztig theory, the pair $(\mathbb{T}, \theta)$
yields a generalized character $R_{\mathbb{T}, \theta}$ of $\mathbb{G}(\mathbb{F}_q) = GL(2, \mathbb{F}_q)$.  We let $C^0(s)$ denote the connected component of the centralizer of $s$ in $\mathbb{G}$.

\begin{proposition}\label{carterformula}{\cite[Proposition 7.5.3]{carter}}
If $s \in \mathbb{G}^{\Phi}$ is semisimple, then $$R_{\mathbb{T}, \theta}(s) = \frac{\epsilon_{\mathbb{T}} \epsilon_{C^0(s)}}{|\mathbb{T}^{\Phi}| |C^0(s)^{\Phi}|_p} \sum_{g \in \mathbb{G}^{\Phi} : g^{-1}sg \in \mathbb{T}^{\Phi}} \theta(g^{-1} s g).$$
\end{proposition}

Let $\mathbb{T}_s$ denote the split torus in $\mathbb{G}$.
Then $\mathbb{T}$ is obtained from $\mathbb{T}_s$ by twisting by the canonical generator $w$ of the Weyl group.  Then $\mathbb{T}^{\Phi} = \mathbb{T}_s^{w \Phi}$.

\begin{proposition}
$$R_{\mathbb{T}, \theta}(s) =  \displaystyle\sum_{i=0}^1 \theta(\upsilon^i(s))$$ for all regular semisimple s in $\mathbb{T}^{\Phi}$, where $\upsilon$ is the generator of $\mathrm{Gal}(\mathbb{F}_{q^2}/\mathbb{F}_q)$.
\end{proposition}

\proof
Let $s \in \mathbb{T}^{\Phi}$ be regular semisimple.  Note that if $g \in \mathbb{G}^{\Phi}$ satisfies $g^{-1} s g \in \mathbb{T}^{\Phi}$, then $g \in N_{\mathbb{G}^{\Phi}}(\mathbb{T}^{\Phi})$.  Therefore, $$\sum_{g \in \mathbb{G}^{\Phi} \ : \ g^{-1}sg \in \mathbb{T}^{\Phi}} \theta(g^{-1} s g) = \sum_{g \in N_{\mathbb{G}^{\Phi}}(\mathbb{T}^{\Phi})} \theta(g^{-1} s g) = $$ $$|\mathbb{T}^{\Phi}| \sum_{w \in N_{\mathbb{G}^{\Phi}}(\mathbb{T}^{\Phi}) / \mathbb{T}^{\Phi}} \theta({}^w s) = |\mathbb{T}^{\Phi}|  \displaystyle\sum_{i=0}^1 \theta(\upsilon^i(s))$$
since $W(\mathbb{G}(\mathbb{F}_q),\mathbb{T}(\mathbb{F}_q)) = \mathrm{Aut}(\mathbb{F}_{q^2}/\mathbb{F}_q)$.

Since $s \in \mathbb{T}^{\Phi}$ is regular semisimple, $|C^0(s)^{\Phi}|_p = 1$.  Moreover, $\epsilon_{\mathbb{T}} = \epsilon_{C^0(s)} = -1$.  Therefore, $$R_{\mathbb{T}, \theta}(s) = \epsilon_{\mathbb{T}} \epsilon_{C^0(s)} \displaystyle\sum_{i=0}^1 \theta(\upsilon^i(s)) = (\theta(s) + \theta(\overline{s})). \qedhere $$

Our character formula is defined on the unramified elliptic torus $E^*$.  We wish to show that our character formula agrees with a depth zero supercuspidal character on the set where they are both defined, i.e.  the set $(F^* K_0 \setminus F^* K_1) \cap E^*$.

\begin{lemma}\label{lemmaweird}
$(F^* K_0 \setminus F^* K_1) \cap E^* = F^* A = E^* \setminus F^*(1 + \mathfrak{p}_E) = \{z \in T(F)^{reg} : n(z) = 0 \}$, where $A:= \{p^n u + v \delta : n \geq 0 \}$
\end{lemma}

\begin{proof}
The proof is not difficult.
\end{proof}

\begin{theorem}\label{depthzerocharactersmatchingup}
$F(\tilde\chi)$ agrees with the supercuspidal character of $\pi_{\chi {}_F \mu_{\chi}}$ on $F^* A$.
\end{theorem}

\begin{proof}
Recall that ${}_F \mu_{\chi}$ is the unique quadratic unramified character of $E^*$.  Therefore, we need to show that

$$\frac{\mathrm{deg}(\pi)}{\mathrm{deg}(\sigma)} \left(\frac{\chi(w) + (-1, \zeta) \chi(\overline{w})}{ \tilde\tau(\frac{w - \overline{w}}{2 \delta}) |D(w)|^{1/2})}\right) = \frac{\mathrm{deg}(\pi)}{\mathrm{deg}(\sigma)}(\chi(w) {}_F \mu_{\chi}(w) + \chi(\overline{w}){}_F \mu_{\chi}(\overline{w})) \ \ \forall w \in F^* A.$$

\noindent Let $w \in A$, so $w = p^n u + v \delta, n \geq 0$.  Then $|D(w)| = 1$.  Moreover, $\tilde\tau( \frac{w - \overline{w}}{2 \delta} ) = \tilde\tau(v) = 1$.  But ${}_F \mu_{\chi}(w) = 1 \ \forall w \in A$ since ${}_F \mu_{\chi}$ is unramified.  Therefore, both sides agree on $A$.  Finally, since $\tilde\tau(x) = {}_F \mu_{\chi}(x)^{-1} =  \ \forall x \in F^*$, we have that both sides agree on $F^* A$.
\end{proof}

Note that in constructing the character formula $F(\tilde\chi)$, we have chosen $\Delta^+$ to be the standard positive root.  If we had made the other choice of $\Delta^+$, the denominator in our character formula would include the term $\tilde\tau(\overline{w} - w)$ instead of $\tilde\tau(w - \overline{w})$.  However, because our definition of $\epsilon(\tilde\chi, \Delta^+, \tau)$ includes the term $\epsilon(\Delta^+)$, $F(\tilde\chi)$ remains the same regardless of the choice of positive root.

\subsection{On whether there are two character formulas coming from the same Cartan}

In this section, we show that if the distribution characters of two depth zero supercuspidal representations, both (necessarily) coming from the unramified Cartan, agree on the $n(w) = 0$ range, then the supercuspidal representations are isomorphic.

\begin{theorem}\label{depthzerosamecartans}
Suppose $(E/F, \chi_1)$, $(E/F, \chi_2)$ are admissible pairs such that $F(\tilde\chi_1)(w) = F(\tilde\chi_2)(w)$ on the set $E^* \setminus F^*(1 + \mathfrak{p}_E)$.  Then $\chi_1 = \chi_2^{\upsilon}$ for some $\upsilon \in \mathrm{Aut}(E/F)$.
\end{theorem}

\begin{proof}
An analogous argument as in \cite[pages 73-74]{spice} works here.
\end{proof}

Summing up, we have altogether shown that if $(E/F, \chi)$ is an admissible pair such that $\chi$ has level zero, then there is a unique depth zero supercuspidal representation, $\pi_{\chi {}_F \mu_{\chi}}$, whose character, on the range $\{ z \in T(F)^{reg} : 0 \leq n(z) \leq r/2 \} = \{z \in T(F)^{reg} : n(z) = 0 \}$, agrees with $F(\tilde\chi)$.  There is one minor point here to resolve.  Is there possibly a positive depth supercuspidal representation whose character, on the range $\{z \in T(F)^{reg} : n(z) = 0 \}$, also equals $F(\tilde\chi)$?  The same argument as in the end of section \ref{othercartan} shows that there isn't.

Therefore, combining theorems \ref{depthzerosamecartans} and \ref{depthzerocharactersmatchingup}, we obtain the following result.

\begin{theorem}
The assignment

\begin{eqnarray}
\{ \mathrm{supercuspidal} \ \phi : W_F \rightarrow SL(2,\mathbb{C}) \} & \mapsto & \tilde\chi \in \widehat{T(F)}_{\tau \circ \rho} \mapsto \pi(\tilde\chi) \nonumber
\end{eqnarray}

\noindent from section \ref{setup} is the Local Langlands correspondence for depth zero supercuspidal representations of $\mathrm{PGL}(2,F)$, where $\pi(\tilde\chi)$ is the unique supercuspidal representation whose character, on the range $\{ z \in T(F)^{reg} : 0 \leq n(z) \leq r/2 \} = \{z \in T(F)^{reg} : n(z) = 0 \}$, is $F(\tilde\chi)$.
\end{theorem}

\emph{Acknowledgements}. I would like to thank Jeffrey Adams for directing my attention to this question and for being a tremendous support throughout my career; and Stephen DeBacker for a very helpful conversation during my visit to Michigan.  I would also like to thank Loren Spice for many very helpful conversations and for answering many questions about his work.  I would also like to thank Jeffrey Adler for several very helpful conversations about my work.  I also had several useful conversations with Hunter Brooks and Sean Rostami.

\bibliographystyle{amsplain}

\begin{thebibliography}{9}

\bibitem{adams1}
  \emph{J. Adams},
  Computing Global Characters, preprint, 2011.

\bibitem{adamsvogan}
  \emph{J. Adams and D. Vogan},
  L-Groups, Projective Representations, and the Langlands Classification.
  Amer. Journal of Math. 113 (1991), 45-138.

\bibitem{adrian}
  \emph{M. Adrian},
  A New Construction of the Local Langlands Correspondence for $GL(n,F)$, $n$ a prime, preprint, arXiv:1008.2727.

\bibitem{adrianlansky}
  \emph{M. Adrian and J. Lansky},
  A Real Groups Construction of the depth zero Local Langlands Correspondence for $PGSp(4,F)$, $F$ a $p$-adic field, preprint.

\bibitem{bushnellhenniart}
  \emph{C. Bushnell and G. Henniart},
  The Local Langlands Conjecture for GL(2), A Series of Comprehensive Studies in Mathematics, Volume 335, Springer Berlin Heidelberg, 2006.

\bibitem{bushnellhenniart1}
  \emph{C. Bushnell and G. Henniart},
  The Essentially Tame Local Langlands Correspondence, I,
  J. Amer. Math. Soc. 18 (2005), no. 3, 685--710.

\bibitem{carter}
  \emph{R. Carter},
  Finite Groups of Lie Type: Conjugacy Classes and Complex Characters, John Wiley and Sons Inc, 1993.

\bibitem{debacker}
  \emph{S. DeBacker},
  On Supercuspidal Characters of $GL_{\ell}$, $\ell$ a prime, Ph.D. thesis, University of Chicago, 1997.

\bibitem{kutzko}
  \emph{P. Kutzko},
  Character Formulas for supercuspidal representations of $GL_{\ell}$, $\ell$ a prime.,
  Amer. J. Math. 109 (1987), no. 2, 201--221.

\bibitem{kutzko1}
  \emph{P. Kutzko},
  The Langlands conjecture for $GL(2)$ of a local field.,
   Ann. of Math. (2) 112 (1980), no. 2, 381–412.

\bibitem{moy}
  \emph{A. Moy},
  Local Constants and the Tame Langlands Correspondence,
   American Journal of Math.  108  (1986),  no. 4, 863--929.

\bibitem{moyprasad}
  \emph{A. Moy and G. Prasad},
  Unrefined minimal $K$-types for $p$-adic groups,
   Invent. Math. 116, no. 1-3, 393-408 (1994).

\bibitem{rao}
  \emph{R. Ranga Rao},
  On some explicit formulas in the theory of Weil representation,
   Pacific J. Math.  157  (1993),  no. 2, 335--371.

\bibitem{roe}
  \emph{D. Roe},
  The Local Langlands Correspondence for Tamely Ramified Groups, Ph.D. thesis, Harvard University, 2011.

\bibitem{sallyshalika}
  \emph{P. Sally, J. Shalika},
  Characters of the discrete series of representations of ${\rm SL}(2)$ over a local field.,
   Proc. Nat. Acad. Sci. U.S.A. 61 1968 1231--1237.

\bibitem{shimizu}
  \emph{H. Shimizu},
  Some examples of new forms,
  J. Fac. Sci. Univ. Tokyo  24  (1977),  no. 1, 97-113.

\bibitem{silberger}
  \emph{A. Silberger},
  $PGL(2)$ over the $p$-adics: its representations, spherical functions, and Fourier analysis. Lecture Notes in Mathematics, Vol. 166 Springer-Verlag, Berlin-New York 1970 vii+204 pp.

\bibitem{spice}
  \emph{L. Spice},
  Supercuspidal Characters of $SL_{\ell}$ over a $p$-adic field, $\ell$ a prime, Amer. J. Math. 121 no. 1, 51–100

\end{thebibliography}

\end{document}